\documentclass[12pt, reqno]{amsart}

\usepackage{amssymb, amsmath, amsthm,stmaryrd}
\usepackage[backref]{hyperref}
\usepackage{amscd}   
\usepackage{fullpage}
\usepackage[all,cmtip]{xy} 
\usepackage{graphicx}
\usepackage{mathrsfs}
\usepackage{multirow}
\DeclareFontEncoding{OT2}{}{} 

\newtheorem{lemma}{Lemma}[section]
\newtheorem{theorem}[lemma]{Theorem}
\newtheorem{proposition}[lemma]{Proposition}
\newtheorem{prop}[lemma]{Proposition}
\newtheorem{cor}[lemma]{Corollary}
\newtheorem{conj}[lemma]{Conjecture}

\newtheorem{claim*}{Claim}
\newtheorem{thm}[lemma]{Theorem}

\theoremstyle{definition}
\newtheorem{defn}[lemma]{Definition}
\newtheorem{remark}[lemma]{Remark}

\newtheorem{example}[lemma]{Example}
\newtheorem{notation}[lemma]{Notation}
\newtheorem{question}[lemma]{Question}
\newtheorem{hypothesis}[lemma]{Hypothesis}

\def\P{\mathbb{P}}

\newcommand{\A}{{\mathbb A}}

\newcommand{\PP}{{\mathbb P}}
\newcommand{\C}{{\mathbb C}}
\newcommand{\F}{{\mathbb F}}

\newcommand{\Z}{{\mathbb Z}}
\newcommand{\N}{\mathbb N}
\newcommand{\nats}{{\mathbb N}}
\newcommand{\ints}{{\mathbb Z}}

\newcommand{\complex}{{\mathbb C}}
\newcommand{\aff}{{\mathbb A}}
\newcommand{\proj}{{\mathbb P}}

\newcommand{\calC}{{\mathcal C}}

\DeclareMathOperator{\Aut}{Aut}

\DeclareMathOperator{\Spec}{Spec}

\DeclareMathOperator{\Frac}{Frac}

\numberwithin{equation}{section}
\numberwithin{table}{section}

\newcommand{\ol}[1]{\overline{#1}}
\newcommand{\mc}[1]{\mathcal{#1}}

\begin{document}

\title{Abhyankar's Conjectures in Galois theory:\\ current status and future directions}

\author[]{David Harbater}
\address{Department of Mathematics, University of Pennsylvania, Philadelphia, PA 19104-6395}
\email{harbater@math.upenn.edu}

\author[]{Andrew Obus}
\address{Department of Mathematics, University of Virginia, Charlottesville, VA 22904}
\email{obus@virginia.edu}

\author[]{Rachel Pries}
\address{Department of Mathematics, Colorado State University, Fort Collins, CO 80523}
\email{pries@math.colostate.edu}

\author[]{Katherine Stevenson}
\address{Department of Mathematics, California State Northridge, Northridge, CA 91330}
\email{katherine.stevenson@csun.edu}

 \makeatletter
\let\@wraptoccontribs\wraptoccontribs
\makeatother

\dedicatory{Dedicated to Yvonne Abhyankar} 
\thanks{
We would like to thank Sudhir Ghorpade and Avinash Sathaye for inspiring us to write this survey.
Harbater was partially supported by NSF FRG grants DMS-1265290 and DMS-1463733, and by NSA grant H98230-14-1-0145.
Obus was partially supported by NSF FRG grant DMS-1265290 and
NSF grant DMS-1602054.
Pries was partially supported by NSA grant 131011 and NSF grant DMS-15-02227.}


	\begin{abstract}
	In this paper, we survey the major contributions of Abhyankar to the development of the theory of fundamental groups and 
	Galois covers in positive characteristic.  
	We first discuss the current status of four conjectures of Abhyankar about Galois covers in positive characteristic.  
Then we discuss research directions inspired by Abhyankar's work, 
	including many open problems.
	\end{abstract}

\maketitle

\section{Introduction}

In this paper, we discuss S.S.~Abhyankar's contributions to Galois theory in positive characteristic, and directions emerging from his work on this topic.
As in the classical case over the complex numbers, 
extensions of function fields correspond to branched covers, and the Galois groups of those field extensions are quotients of the associated fundamental groups.  But here, instead of considering function fields and covers of Riemann surfaces or other complex varieties, one works with algebraic varieties over a field of characteristic $p$.  Abhyankar's investigation of covers of varieties in positive characteristic led to the discovery of many 
unusual examples and phenomena that simply do not exist in characteristic $0$, and which were later studied further by many authors.

In fact, these phenomena emerge even when studying covers of the affine line.  Viewed as a one-dimensional complex algebraic variety, $\mathbb C$ 
is the ``complex line'', and as in topology it is simply connected.  Hence it has no non-trivial unramified covers.  But if $\mathbb C$ is replaced by an an algebraically closed field $k$ of characteristic $p$, 
the line has many unramified covers, and so it is far from simply connected.  For example, in the $x,y$-plane over $k$, the curve $y^p-y-x=0$ is a $p$-to-one unramified cover of the $x$-line, since 
$\frac{\partial}{\partial y}(y^p-y-x)$ never vanishes.  Abhyankar's explorations went well beyond this Artin-Schreier example, whose Galois group is cyclic of order $p$; he found that the fundamental group of the affine line in characteristic $p$ is not even solvable!  His explorations also included higher dimensional varieties in characteristic $p$, and also varieties over finite fields, where arithmetic considerations arise.

In searching to explain the structure underlying the examples he found, Abhyankar 
developed a philosophy emerging from two themes.  The first, which limits the possible phenomena, is that
the theory of covers of varieties should be the same in characteristic $p$ as in characteristic $0$, after eliminating
the prime $p$ from the situation.  The second, which pulls in the opposite direction, is a kind of Murphy's Law: whatever can happen in characteristic $p$, will happen.

This philosophy was in part realized in later work by Serre on embedding problems and by Grothendieck on the prime-to-$p$ fundamental group.
However, turning this philosophy into precise statements about Galois covers turned out to be surprisingly subtle.  
Abhyankar stated several conjectures about 
Galois theory in positive characteristic and gradually refined these conjectures over the course of his career.

\subsection{Abhyankar's Conjectures}
We begin by mentioning in particular four of Abhyankar's conjectures that emerged from this philosophy.  These conjectures, and the current state of knowledge about them, will be discussed in greater detail in  
Sections \ref{SConj1}--\ref{dShighdim} below.

\subsubsection{Galois covers of affine curves.} \label{intro Gal cov}

Classically, if $X$ is a smooth projective complex curve (i.e., a compact Riemann surface) of genus $g \ge 0$, and if $B \subset X$ is a set of $r>0$ distinct points,  then the fundamental group of $U = X - B$ is isomorphic to the free group on $2g+r-1$ generators.  The quotients of $\pi_1(U)$ are precisely the covering groups of Galois (i.e., regular) covering spaces of $U$, also called the Galois groups of those covers. 
By the structure of $\pi_1(U)$, the Galois groups that arise are precisely the groups that have a generating set of size at most $2g+r-1$.
In characteristic $p$, though, there are more covers, and more covering groups, over a given affine curve $U = X - B$.  For example, if $U$ is the affine line in characteristic $p$, then $2g+r-1 = 0$, yet $\mathbb Z/p\mathbb Z$ is a Galois group over $U$.   Abhyankar's philosophy says that the Galois groups of order prime to $p$ over $U$ should be the same as those that arise in the analogous situation for complex algebraic curves; and that this should be the only constraint on the types of Galois groups that arise.  

More precisely, Abhyankar posed the following conjecture, where
$p(G)$ denotes the subgroup of a finite group $G$ that is generated by the $p$-subgroups of $G$, and 
$k$ is an algebraically closed field of characteristic $p >0$.

\begin{conj} [Abhyankar's Conjecture for affine curves.] \label{Cintro1}
Let $X$ be a smooth projective curve of genus $g$ defined over $k$.  
Let $B$ be a non-empty set of points of $X$ having cardinality $r$ and let $U=X-B$.
Then a finite group $G$ is the Galois group of an unramified cover of $U$ if and only if $G/p(G)$ has a generating set of size at most $2g+r-1$.
\end{conj}

In Section~\ref{SConj1} we describe the proof of Conjecture \ref{Cintro1}, which involves work of 
Serre, Raynaud, and Harbater.  The proof and progress leading up to it involved the development of 
a significant amount of new theory, including results on embedding problems for fundamental groups, formal patching of modules, and 
semi-stable reduction of Galois covers.

In particular, consider the situation when $X$ is the projective line and $B$ is the point at infinity.
Conjecture \ref{Cintro1} states that a finite group $G$ is the Galois group of an unramified cover of the affine line $U$ 
if and only if $G$ is generated by its $p$-subgroups.  The groups satisfying this property are called 
{\it quasi-$p$} groups.

\subsubsection{Inertia groups of Galois covers of the 
affine line.}

Classically, given a covering space $V$ of $U = X  -  B$, where $X$ is a compact Riemann surface and $B \subset X$ is finite, there is a Riemann surface $Y$ that compactifies $V$, together with a map 
$Y \to X$ that extends the given map $V \to U$.  If $V \to U$ is Galois with group $G$, then so is $Y \to X$.  But $Y \to X$ is a branched cover, 
with the branch locus contained in $B$ and the ramification locus contained in $Y - V$.
The map $Y \to X$ is not evenly covered over the points of $B$; instead, 
for each point $P$ of $B$, the map $Y \to X$ is conformally equivalent to the map $y^n=x$ near each point $Q$ over $P$, where $n$ is the ``number of sheets'' of the cover that coalesce at $Q$.  If $n>1$, then the subgroup of $G$ that fixes $Q$ is non-trivial; in fact it is cyclic of order $n$, whose generator is given locally near $Q$ by sending the coordinate $y$ to $\zeta_n y$.  (Here $\zeta_n = e^{2\pi i/n}$, a primitive $n$-th root of unity.)  This group is called the {\em inertia group} at $Q$.  If $Q'$ is another point over $P$, then the inertia groups at $Q$ and at $Q'$ are conjugate in the Galois group.

But over an algebraically closed field of characteristic $p$, an inertia group need not be cyclic if its 
order is divisible by $p$.  Instead, it is a semi-direct product of a $p$-group with a cyclic group of order prime to $p$; this is also referred to as an {\em extension of a cyclic group of order prime-to-$p$ by a $p$-group}.  Abhyankar's next conjecture aims to give necessary and sufficient 
conditions on a subgroup of the Galois group of a cover over $k$ to be an inertia group, in the case of the affine line.  As above, $k$ is an algebraically closed field of characteristic $p$.

\begin{conj}[Abhyankar's Inertia Conjecture] \label{Cintro2}
Let $G$ be a finite quasi-$p$ group.  Let $G_0$ be a subgroup of $G$
which is an extension of a cyclic group of order prime-to-$p$ by a $p$-group $G_1$.  
Then $G_0$ occurs as an inertia group for a $G$-Galois cover of the projective line branched only at $\infty$
if and only if the conjugates of $G_1$ generate $G$.
\end{conj}

The motivation here is that for Galois branched covers $Y \to X$ of the Riemann sphere, all the inertia groups together generate the Galois group $G$.  This is because the intermediate cover of $X$ corresponding to the subgroup of $G$ that they generate would then be an unramified cover of the Riemann sphere and hence trivial.  Thus if we pick one point of $Y$ over each branch point in $X$, then the conjugates of the corresponding inertia groups generate $G$.  In characteristic $p$, this observation combined with the two parts of Abhyankar's philosophy suggest the above conjecture; see also Section~\ref{SSConj2}.

The status of Conjecture \ref{Cintro2} is still open 
(although the ``only if" direction is known by Grothendieck's result on the tame fundamental group).  
In Section \ref{SConj2}, we discuss the few situations in which it has been verified.
Section \ref{Sconstraint} contains a new result about the interaction between the structure of the inertia group, the ramification filtration, and the genus of a quotient of the cover.

\subsubsection{Galois covers of curves defined over finite fields.}
If $\ell$ is a finite field, then its algebraic closure $\bar \ell$ is an infinite Galois extension of $
\ell$ whose finite subextensions all have cyclic Galois groups over $\ell$.  This suggests that replacing the algebraically closed field of constants $k$ by a finite subfield $\ell$ adds a generator to the fundamental group of an affine curve, somewhat like the effect of deleting a point.  This perspective motivated:

\begin{conj}[Abhyankar's Affine Arithmetical Conjecture] \label{Cintro3}
A finite group $A$ occurs as the Galois group of an unramified cover of the affine line over $\F_p$ 
if and only if 
it occurs as the Galois group of an unramified cover of $\A^1_k - \{0\}$
(in other words, if and only if $A/p(A)$ is cyclic).
\end{conj}

The status of Conjecture \ref{Cintro3} is still open.  We discuss this further in Section \ref{SConj3}, along with a related result.  

\subsubsection{Galois covers of higher-dimensional affine varieties. }

Abhyankar also made conjectures about Galois covers of higher dimensional affine varieties over algebraically closed fields of characteristic $p$.  In particular, he considered the following higher dimensional spaces:
\begin{enumerate}
\renewcommand{\theenumi}{\roman{enumi}}
\renewcommand{\labelenumi}{(\roman{enumi})}\item
$X = \Spec\bigl(k[[x_1,\dots,x_n]][x_1^{-1},\dots,x_t^{-1}]\bigr)$, where $n \ge 2$ and $t \ge 1$.
\item
$X$ is the complement of a set of $t+1$ hyperplanes in general position in $\P^n_k$, where $n \ge 2$ and $t \ge 1$  Here, ``general position'' means that the union of these hyperplanes is a normal crossing divisor.
\end{enumerate}

These situations had been studied by Zariski in the case of complex varieties; see Section~\ref{fund gps history} below for a discussion of this history and background.  For the characteristic $p$ analog, 
the statements of Abhyankar's conjectures are more complicated than those above, see Section \ref{dShighdim} for details.
In fact, some modification of the original versions of the conjectures turned out to be necessary. 

\subsection{Current research directions}

Abhyankar's conjectures in Galois theory continue to motivate new developments.
In addition to the two open conjectures highlighted above, there are many questions about 
Galois covers of varieties that arise naturally when studying his work.  
In the second part of the paper, we describe several current research directions in Galois theory that 
are inspired by Abhyankar's research and which may contribute to the solution of some of his problems.
These include:
projective curves (Section \ref{Sproj});
ramification theory (Section \ref{Sram}); and
lifting problems (Section \ref{aSoort}).
There are many other current research directions in Galois theory which we will not address here.
Also, since our focus is on Abhyankar's work in Galois theory, we do not discuss his work on other topics, such as resolution of singularities, valuation theory, computational algebraic geometry, and the Jacobian conjecture, except tangentially.  For a brief overview of the topics on which Abhyankar worked, we refer the reader to the obituary article that appeared in the Notices of the AMS \cite{Notices}.

\section{Algebraic fundamental groups}

\subsection{History and motivation} \label{fund gps history}

Abhyankar's advisor, Oscar Zariski,
had studied the fundamental groups of complex varieties, and in particular the complements of curves in the complex projective plane \cite{Zar29}.  
But whereas this study could rely on the definition of fundamental groups in terms of homotopy classes of loops, another approach is needed in order to work with varieties over more general fields.  

Namely, one can consider the inverse system of finite unramified Galois covers, and the corresponding inverse system $\pi(U)$ of Galois groups.  
In the case of varieties over $\mathbb C$, the inverse limit of these groups is the profinite completion of the classical (topological) fundamental group.  
Over a more general field, one can consider this inverse system of groups to be a replacement for the fundamental group.  
This is what was done by Abhyankar in \cite{Abh59}, where he studied the complements $U$ of curves $C$ in the projective plane over an arbitrary algebraically closed field of characteristic $p \ge 0$.  

In that paper, Abhyankar also considered the subsystem $\pi'(U)$ of Galois groups of covers that are only tamely ramified over $C$.  In particular, generalizing a result in \cite{Zar29}, he showed in Section~13 of \cite{Abh59} that  the groups in $\pi'(U)$ are abelian if $C$ is a divisor with strict normal crossings, i.e.\ a union of smooth curves that intersect transversally.  
He also stated the assertion in the case of more general normal crossings (allowing the components of $C$ to be nodal curves), and this was later shown by Fulton \cite{Ful80}; cf.\ also \cite[Remark~4.3]{HarvdP02}.

Abhyankar's motivation for studying fundamental groups from this perspective arose from a rather different problem that he had considered in his thesis, related to resolution of singularities.  
His advisor, Zariski, had suggested to him that he try to prove the local uniformization theorem for algebraic surfaces over an arbitrary algebraically closed field, by generalizing the proof that was given by Jung \cite{Jung08} in the complex case.  
For complex surfaces, Zariski had proven resolution of singularities \cite{Zar39} using Jung's ideas (by way of Walker \cite{Wal35}); and it was hoped that by carrying over Jung's proof to more general fields, resolution of singularities for surfaces could be proven over such fields too.  
But as Abhyankar liked to put it, this effort was the ``failure'' part of his thesis, because Jung's proof did not carry over to varieties over fields of characteristic $p >0$ \cite{Abh55}.  
Instead, Abhyankar was able to prove local uniformization over such fields in another way \cite{Abh56}, and that led to further work of his on resolution of singularities, including his proof of resolution for rather general surfaces (\cite{Abh65}, \cite{Abh69}).

His attempt to generalize Jung's proof, however, led Abhyankar to study unramified covers of varieties.  
He had observed that Jung's proof relied on two key facts about branched covers $V \to W$ of complex surfaces, with $W$ smooth and $V$ normal: that $V$ is smooth at each point lying over over a smooth point of the branch locus; 
and that for Galois covers, the local Galois group is abelian at any point of the ramification locus over an ordinary double point of the branch locus.  
But he discovered that these properties no longer hold in characteristic~$p$.  

In fact, as Abhyankar observed in \cite[Section~3]{Abh55}, 
it is possible for $V$ to be singular at a point $P$ over a smooth point of the branch locus, and in fact for the ramification locus to have multiple branches at $P$.  
He also observed there that it is possible for the inertia group to be non-solvable at a ramification point $P$ lying over an ordinary double point of the branch locus, or even over a smooth point of the branch locus.  
As Abhyankar later related in \cite[Section~7]{Abh01}, this was quite contrary to Zariski's expectations.  In \cite[Section~1]{Abh57}, Abhyankar referred to these characteristic $p$ phenomena as ``local splitting of a simple branch variety by itself.''  (This phenomenon later caused difficulty in an early attempt by the first author \cite{Har91} to prove Abhyankar's Conjecture on the fundamental group of the affine line in characteristic $p$; that was bypassed a few years later by Raynaud, whose proof avoided this problem through the use of semi-stable models \cite{Ray94}.)
 
In order to understand these phenomena better, Abhyankar restricted branched covers of surfaces $V \to W$ to curves $C$ in $W$, and then studied the resulting branched covers of curves.  This led to his study of the fundamental group of the affine curve $U$ given by the complement of the branch locus in $C$.  In addition to considering the inverse system $\pi(U)$ mentioned above, he also considered the set of finite groups $\pi_A(U)$ that arise in this inverse system; i.e., which are Galois groups of finite unramified Galois covers of $U$.  He said \cite[Section~4.2]{Abh57} that in the case of complex curves, this set can be regarded as ``a good algebraic approximation to [the topological] $\pi_1$''; and that in the general case, 
``eventually one may have to consider the Galois group'' of the compositum of all the extensions of the function field that are unramified over $U$.  This in fact came to pass, in the approach of Grothendieck using the notion of the \'etale fundamental group.  We discuss this next.

\subsection{\'Etale fundamental groups}\label{Sfundgroups}
In the classical situation of complex algebraic varieties, covering spaces can be defined as usual in topology, in terms of points having evenly covered neighborhoods.  But those neighborhoods are taken in the usual complex metric topology, not in the Zariski topology.  In contrast, a more general field is not equipped with a metric, and so only the Zariski topology is available.  But that topology, whose closed sets are generated by the subvarieties of the given space, is extremely coarse.  As a result, there are never evenly covered neighborhoods except in trivial cases.  

To remedy that, a different characterization of covering spaces $Y \to X$ is needed, which carries over well to more general fields.  The one that works uses the notion of being {\em \'etale}.  This condition says that the hypothesis of the implicit function theorem is satisfied.  That is, if $Y$ is locally given over $X$ at a point by equations $f_1=\cdots=f_n=0$ in variables $y_1,\dots,y_n$, then the Jacobian matrix $(\partial f_i/\partial y_j)$ is invertible.  Here the derivatives are taken formally, without reference to a metric; one simply {\em defines} the derivative of $\sum a_iy^i$ to be $\sum ia_i y^{i-1}$.  If $Y \to X$ satisfies the Jacobian condition at every point, it is said to be unramified, or \'etale.  (For more about the notion of \'etale, see \cite[Section~III.5]{redbook}.  According to a footnote there: ``The word apparently refers to the appearance of the sea at high tide under a full moon in certain types of weather.'')  A shorter and equivalent condition, using commutative algebra, is that  $Y \to X$ is flat and unramified.  Here flatness is automatic under mild hypotheses, for example if $X$ is a regular curve or surface, and $Y$ is normal.

To correspond to a covering space, a finite-to-one map $Y \to X$ also should not ``miss any points'' over $X$ (i.e., it should be topologically proper).  This is equivalent to assuming that the map is {\em finite}, in the sense of the ring of polynomial functions on $Y$ being a finitely generated module over the ring of such functions on $X$.  Like \'etale, this condition is purely algebraic.  Thus one can speak of a (finite) \'etale cover $Y \to X$ even when working over a quite general base field; and if that field is $\mathbb C$ then this is equivalent to being a covering space of finite degree.

In order to define fundamental groups in terms of covering spaces, one classically invokes the universal cover, which is a manifold if the given space is.  But a given algebraic variety generally does not have another variety as its universal cover.  Instead, one works with the full inverse system of finite 
\'etale covers $Y \to X$,
which is well defined once one picks a base point $P$ on $X$ and a point on each $Y$ above $P$.
The automorphism group of this inverse system is then called the {\em \'etale fundamental group} $\pi_{1,{\rm et}}(X)$.  As in topology, the isomorphism class of this profinite group does not depend on the choice of $P$.  Alternatively, this group can be regarded as the inverse limit of the Galois groups of the Galois \'etale covers $Y \to X$ in the inverse system of covers.  

In the case of a smooth affine curve $U$, say with function field $F$, there is a unique smooth projective model $X$ of $F$, and $U = X - B$ for some finite set $B$ of closed points.  In this situation, $\pi_{1,{\rm et}}(U)$ can be identified with the Galois group of the compositum of the finite extensions of $F$ that are ramified only over $B$; this recaptures the description proposed by Abhyankar.  The finite quotients of $\pi_1(X)$ are precisely the groups in Abhyankar's $\pi_A(X)$; but $\pi_1(X)$ also contains information about how these finite groups ``fit together''.

In the case that $X$ is a complex variety, this group is the profinite completion of the topological fundamental group, and thus ``approximates'' the classical $\pi_1$, as suggested by Abhyankar.  When working algebraically, especially over fields other than the complex numbers, we have only the \'etale fundamental group at our disposal, and one generally writes $\pi_1(X)$ for short instead of $\pi_{1,{\rm et}}(X)$.  There are other ways as well of defining this group, e.g.\ via fiber functors; see \cite{SGA1} for more details.

It was expected that the algebraic fundamental group of a variety over an algebraically closed field $k$ of characteristic zero should not depend on the choice $k$; and so the group should be the same as over $\C$, viz.\ the profinite completion of the topological fundamental group.  This was a central result shown in \cite{SGA1}.  Namely, according to 
\cite[XIII, Corollary 2.12]{SGA1}, if $X$ is a smooth projective $k$-curve of genus $g$ and $B \subset X$ consists of $r$ closed points,  
then the fundamental group $\pi_1(X-B)$ is isomorphic to the profinite completion of the topological 
fundamental group of a Riemann surface of genus $g$ with $r$ punctures.  
This is the profinite group on generators $a_1,\dots,a_g,b_1,\dots,b_g,c_1,\dots,c_r$ subject to the single relation $\prod_{i=1}^g [a_i,b_i] \prod_{j=1}^r c_j = 1$.  
If $r > 0$, then $\pi_1(X - B)$ is a free profinite group on $2g+r-1$ generators. 
Thus every group generated by $2g+r-1$ elements is a quotient of $\pi_1(X-B)$. 
Also, the freeness implies that all {\em embedding problems} can be solved:  
given an $H$-Galois cover $\psi: Y \to X$, a group $G$ generated by $2g+r-1$ elements, and a surjection 
$G \twoheadrightarrow H$, there exists a $G$-Galois cover of $X$ that factors through $\psi$.   Equivalently, every $H$-Galois extension of $F$ that is unramified over $B$ can be embedded in a $G$-Galois extension of $F$ that is unramified over $B$, where $F$ is the function field of $X$.

There are further results in \cite{SGA1} concerning the case of non-zero characteristic.  For a prime $p$, we can consider the prime-to-$p$ fundamental group $\pi_1^{p'}$ of a variety, in which one instead takes the inverse system of \'etale covers whose Galois group has order prime to $p$.  This is the maximal prime-to-$p$ quotient of the full \'etale $\pi_1$.  It was shown in \cite[XIII, Corollary 2.12]{SGA1} that if $k$ is an algebraically closed field of characteristic $p \ne 0$, and $U$ is obtained by deleting $r>0$ points from a smooth projective curve of genus $g$, then $\pi_1^{p'}(U)$ is isomorphic to $\pi_1^{p'}$ of the analogous curve over $\C$.  This is consistent with what Abhyankar had proposed in \cite[Section~4.2]{Abh57}, concerning which finite groups are Galois groups of unramified covers of affine curves in characteristic $p$.

There is also a partial result for {\em tamely ramified} covers, i.e.\ those covers in characteristic $p$ whose ramification indices are all prime to $p$.  Namely, the above theorem also asserted that the corresponding fundamental group $\pi_1^t(U)$ is a quotient of the \'etale fundamental group of the analogous complex curve (and in fact is a quotient of the {\em $p$-tame} fundamental group over $\C$, which is the group obtained by taking complex covers of ramification index prime to $p$).  This theorem thus shows that $\pi^t(X-B)$ and $\pi_1^{p'}(X-B)$ are finitely generated as profinite groups, which in turn implies that they
are determined by their finite quotients \cite[Prop.\ 15.4]{FJ}.  This will be discussed further in Section \ref{Sproj}.  But the full structure of $\pi_1^t(U)$ remains unknown.

One can also consider the pro-$p$ fundamental group $\pi_1^p$,
where one restricts to the Galois groups that are $p$-groups. 
By a result of Shafarevich \cite{shaf}, if $X$ is a projective $k$-curve then
$\pi_1^p(X)$ is a free pro-$p$ group on $s_X$ generators, where $X$ has $p$-rank equal to $s_X$.   
If $C$ is an affine $k$-curve, then the pro-$p$ fundamental group $\pi_1^p(C)$ is infinitely generated, also by a result of Shafarevich \cite{shaf}.

\subsection{Decomposition and inertia groups} \label{Sdecomptame}

In characteristic $p \ne 0$, the structure of ramification is more complicated than in characteristic zero, as discussed in Section~\ref{fund gps history} above.  
We describe this structure in some detail below, for the sake of later sections, where it is needed in the study of fundamental groups and Galois covers.

Let $\phi:Y \to X$ be a Galois cover of curves with group $G$.  
Let $B$ be the branch locus of $\phi$ and let $\xi \in B$.  Let $r_{\xi}$ be the number of
ramification points $\eta \in \phi^{-1}(\xi)$.  
Let $\pi_{\eta}$ be a uniformizer (local parameter) of $Y$ at $\eta$.

\begin{defn} \label{decomp inertia def}
\begin{enumerate}

\item The {\it decomposition group} $D$ of $\phi$ at $\eta$ is the
subgroup $D \subset G$ which fixes $\eta$ (i.e.\ $D = \{g \in G: g(\eta)=\eta\})$.

\item The {\it inertia group} $I$ of $\phi$ at $\eta$ is the
subgroup $I \subset D$ which acts by the identity on $\hat{{\mathcal O}}_{Y, \eta}/\pi_{\eta}$.
(If $k$ is an algebraically closed field and $\eta \in Y$ is a $k$-point, then the inertia 
group and decomposition group at $\eta$ are the same.)

\item The cover $\phi$ is {\it wildly ramified} at $\eta$ if $p$ divides the order of the inertia group at $\eta$.

\item The cover $\phi$ is {\it tamely ramified} at $\eta$ if $p$ does not divide the order of the inertia group at $\eta$.  (In characteristic zero, all ramification is tame.)
\end{enumerate}
\end{defn}

The decomposition and inertia groups at a point $\eta \in Y$ over $\xi \in X$ can be studied by means of the corresponding extensions of complete local rings at those points, or equivalently in terms of their fraction fields, which are complete discretely valued fields.  Namely, the Galois group of this field extension is the decomposition group at $\eta$, which is the same as the inertia group if the ground field $k$ is algebraically closed.  See \cite{Serr}, Chapter~II, Section~2.

Tame inertia groups are thus cyclic, 
while wild inertia groups can be more complicated (though still solvable).  The prime-to-$p$ part of inertia can 
be eliminated after a tame pullback, by the following result of Abhyankar (taking $n=m)$, generalizing behavior in characteristic zero.

\begin{lemma}[Abhyankar's Lemma, \cite{SGA1}, XIII, Proposition 5.2] \label{Lablemma}
Suppose $\phi:Y \to X$ is a $G$-Galois cover branched over $x$ with
inertia group $I$ of the form $P \rtimes \Z/m$, where $|P|$ is a
power of $p$ and $(m, p) = 1$.
Let $\psi: X' \to X$ be a $\Z/n$-Galois cover branched over $x$ and suppose that $\phi$ and $\psi$ are linearly disjoint.
Then $\psi^* \phi$ is a $G$-Galois cover and the inertia group above points in $\psi^{-1}(x)$ is of the form $P \rtimes \Z/(m/{\rm gcd}(m,n))$.
\end{lemma}

\subsubsection{Higher ramification groups} \label{Shigherram}

We now focus on the case that $\phi$ is wildly ramified at some point $\eta$, in which case the inertia groups need not be cyclic.  To understand the structure of these more complicated groups,
one introduces the notion of ``higher ramification groups.''  
The following material can be found in \cite[IV]{Serr}.  While the discussion there is stated in terms of 
extensions of complete discretely valued fields, it applies as well to covers of curves by the comment above.  

\begin{defn} \label{Dlowerjump}  \cite[IV, Section 3]{Serr}
\begin{enumerate}
\item The {\it lower numbering filtration of higher ramification groups} $\{I_i:
i \in \N^+\}$ of $I$ at $\eta$ is defined as follows:  If $g \in I$ then
$g \in I_i$ if and only if $g(\pi_{\eta}) \equiv \pi_{\eta} 
\mod{\pi_{\eta}^{i+1}}$. 

\item If $g \in I$ then the {\it lower jump} for $g$ of $\phi$ at
$\eta$ is the integer $j$ such that $g \in I_j -I_{j+1}$;
in other words, the lower jump for $g$ is $j$ if and only if 
$j+1$ is the valuation of $g(\pi_{\eta})-\pi_{\eta}$.
The filtration $\{I_i\}$ has a {\it lower jump} at $j$ if 
$I_j \not = I_{j+1}$. 
\item If $|I|=pm$ with $p \nmid m$, then the {\it lower jump} of $\phi$ is defined to
be the lower jump of an element of order $p$ in $I$.
\end{enumerate}
\end{defn}

For Galois covers $Y \to X$ in characteristic zero, and more generally in the tamely ramified case, the classical Riemann-Hurwitz formula relates the genus of $Y$ to the genus of $X$, the degree of the cover, and the ramification indices, which are the orders of the (cyclic) inertia groups.  If we allow wildly ramified covers, there is the following generalization, using higher ramification groups:

\begin{lemma} \label{1SRH}
Suppose $\phi$ is as above.
\begin{enumerate}
\item \label{inertia structure part}
The inertia group $I$ of $\phi$ at $\eta$ is of the
form $I=P \rtimes \Z/m$ where $|P|$ is a power of $p$ and $(m,p)=1$.  
The higher ramification group $I_i$ is a $p$-group for $j \geq 1$.

\item (Riemann-Hurwitz formula) In the above situation,
\[2g_Y-2=\deg(\phi)(2g_X-2) + \sum_{\xi \in B} R_{\xi},\]
where \[R_{\xi}=r_{\xi}\sum_{i=0}^{\infty}(|I_i|-1) \ {\rm and} \ r_\xi=\deg(\phi)/|I|.\]

\item In particular, if $\phi:Y \to \PP_k$ is a Galois cover
branched only at $\infty$ over which it has inertia $I \simeq \Z/p \rtimes
\Z/m$ and jump $j$, then 
\[2g_Y-2=\deg(\phi)(2g_X-2)+ r_{\infty}(mp-1 + (p-1)j).\] 
\end{enumerate} 
\end{lemma}

\begin{proof}
The first item is found in \cite[IV, Corollaries 3 \& 4]{Serr} and the others follow from 
\cite[IV, Proposition 4]{Serr} combined with the standard Riemann-Hurwitz formula.
\end{proof}

There is also an {\it upper numbering filtration of higher ramification
groups} with upper jumps and an explicit formula of Herbrand relating
the upper and lower filtrations \cite[IV, Section 3]{Serr}.  
The upper jumps are preserved for
quotients (\cite[IV, Proposition 14]{Serr}) and the lower jumps are preserved for
subgroups (\cite[IV, Proposition 2]{Serr}).
In the simplest case that $I \simeq \Z/p \rtimes \Z/m$ has lower jump $J$,
the upper jump $\sigma$ satisfies $\sigma=J/m$.  In general, there is

\begin{defn} \label{Dherbrand}
If $j$ is a lower jump in the ramification filtration, 
then $\varphi(j)=|I_0|^{-1}\sum_{i=1}^j|I_i|$ is an {\it upper jump}.
\end{defn}

Using Definition \ref{Dherbrand}, one can precisely define 
the upper numbering filtration of higher ramification groups using linear interpolation between the values of the upper jumps, 
but this will not be needed in this paper.  

\section{Abhyankar's Conjecture on affine curves} \label{SConj1}

\subsection{History and motivation}\label{SAbhHist}

In \cite[Section~4.2]{Abh57}, Abhyankar stated his conjecture on the fundamental group of affine curves over an algebraically closed field of finite characteristic.  He first posed it rather informally as follows.  Let $p(G)$ denote the subgroup of a finite group $G$ that is generated by its $p$-subgroups.  Note that $p(G)$ is characteristic in $G$ and thus normal in $G$ as well.

\begin{conj} [Informal form of Abhyankar's Conjecture] \label{dInformalAC} Let $U$ be an affine variety over algebraically closed field of characteristic $p>0$ and let $G$ be a finite group.  Then $G$ is the Galois group of an unramified Galois cover of $U$ if and only if $G/p(G)$ is the Galois group of an unramified Galois cover of ``the corresponding variety in characteristic zero.''
\end{conj}

Of course, it needs to be explained what is meant by that last phrase, and Abhyankar pointed out that there need not always be such a ``corresponding variety.'' But as he explained by examples, if $U$ is an affine curve in characteristic $p$ of type $(g,r)$, i.e.\ the complement of $r>0$ points in a smooth projective curve of genus $g$, then the corresponding variety would be an affine curve of type $(g,r)$ in characteristic zero.  And if $U$ is the complement in $\mathbb P^n$ of $r$ hyperplanes in general position, then the corresponding variety would be such a complement in characteristic zero.  (Here all the ground fields are assumed algebraically closed.)  

In the case of affine curves, it is classical that the fundamental group of a complex curve of type $(g,r)$ is free of rank $2g+r-1$ (see Section 
\ref{intro Gal cov}), and this holds when the complex numbers are
replaced by any algebraically closed field of characteristic $0$ (Section \ref{Sfundgroups}).  Thus the finite Galois groups of unramified covers are precisely the finite groups that can be generated by a set of at most $2g+r-1$ elements.  In this key case, the above conjecture can be formulated more precisely as follows:

\begin{conj} [Abhyankar's Conjecture on affine curves]\label{dAbhConj} 
Let $X$ be a smooth projective curve of genus $g$ defined over an algebraically closed field of characteristic $p>0$.  
Let $B$ be a non-empty set of points of $X$ having cardinality $r$ and let $U=X-B$.
Then a finite group $G$ is the Galois group of an unramified cover of $U$ if and only if $G/p(G)$ has a generating set of size at most $2g+r-1$.
\end{conj}

The forward direction of Conjecture \ref{dAbhConj} follows from the work of Grothendieck.  
If there is an unramified cover of $U$ with Galois group $G$, then it has a quotient cover which is an unramified 
cover of $U$ with Galois group $G/p(G)$.  The fact that the order of $G/p(G)$ is prime-to-$p$ implies that 
the quotient cover lifts to characteristic $0$.  As mentioned above, the well-known restrictions on Galois covers in 
characteristic $0$ imply that $G/p(G)$ has a generating set of size at most $2g+r-1$.

\begin{remark}\label{dRproj}
The affine hypothesis in Abhyankar's conjecture is crucial as the
corresponding statement is \emph{not} true for projective curves.
Indeed, if $X/k$ is a projective curve of genus $g$ and
$\text{char}(k) = p$, then $\pi_1(X)$ is a \emph{quotient} of the
group $$\Sigma_g := \langle a_1, b_1, \ldots, a_n, b_n | \prod_{i=1}^n
[a_i, b_i] = 1 \rangle$$ (\cite[XIII, Cor. 2.12]{SGA1}).  So, for instance, $G := (\ints/p)^N$ is not a quotient of $\pi_1(X)$ for large $N$, yet $G/p(G)$ is trivial.
\end{remark}

Conjecture \ref{dAbhConj} says in particular that the finite Galois groups over the affine line are precisely the {\em quasi-$p$ groups}, i.e.\ the finite groups $G$ that are generated by their $p$-subgroups.  This includes, for example, all simple groups of order divisible by $p$.  More generally, Conjecture \ref{dAbhConj} says that every quasi-$p$ group is a Galois group over every affine curve in characteristic $p$.  In addition, it says that the set $\pi_A(U)$ of finite Galois groups over an affine curve of type $(g,r)$ is {\em independent} of the curve $U$, and also of the choice of algebraically closed ground field of characteristic $p$.  

Conjecture \ref{dAbhConj} for an affine curve $U$ can be regarded as having two parts (where we write $\pi_A(U)$ for the set of finite Galois groups over $U$, as in \cite{Abh55} and \cite{Abh57}):  
\begin{enumerate}
\item $G \in \pi_A(U)$ if and only if $G/p(G) \in \pi_A(U)$. \label{dAC char p part}
\item If $U$ has type $(g,r)$ and $G$ has order prime to $p$,
then $G \in \pi_A(U)$ if and only if $G \in \pi_A(U_0)$, where $U_0$ is a complex curve of type $(g,r)$. \label{dAC prime to p part}
\end{enumerate}

The motivation for the converse direction of Conjecture \ref{dAbhConj} came from the examples that Abhyankar had computed in \cite{Abh55} and \cite{Abh57}.  These examples showed that in characteristic $p$ there are non-solvable unramified covers of the affine line; that every finite group is a subgroup of some group in $\pi_A(U)$ if $U$ is the complement of any set of at least three points in the projective line; and that every finite group is a subquotient of some group in $\pi_A({\mathbb A}^1)$.  These examples, together with the phenomena discussed in Section~\ref{fund gps history} above, suggested a form of the Murphy's Law mentioned in the introduction: that every type of cover that cannot be ruled out must occur. 

Two cases of the conjecture were known soon after.  Part~(\ref{dAC prime to p part}) above, on prime-to-$p$ groups, was shown in \cite{SGA1} and in \cite{Popp67}.  At the opposite extreme, the conjecture asserts that all finite $p$-groups
are Galois groups over every affine curve in characteristic $p$.  This case follows from the fact that the pro-$p$ part of the \'etale fundamental group is free, which holds because all elementary abelian $p$-groups are Galois groups (by Artin-Schreier theory), together with the fact that the curve has $p$-cohomological dimension one 
\cite[Expos\'e~X, Th\'eor\`eme~5.1]{SGA4}.

But even knowing these two extreme cases, it remained far from obvious that Conjecture \ref{dAbhConj} was true.  For example, in \cite{Kam82}, Kambayashi asked whether every connected degree $p$ unramified cover of the affine line is Galois.  As noted at \cite[p.~356]{Har84}, an affirmative answer to that question would contradict Abhyankar's Conjecture \ref{dAbhConj}.  

In retrospect, there are other reasons why Conjecture \ref{dAbhConj} is remarkable.
While it implies that $\pi_A$ of an affine curve depends only on its type $(g,r)$, it turns out that almost the opposite is true for the \'etale fundamental group $\pi_1$ of the curve.  Namely, by results of Tamagawa (\cite{Tam99}, \cite{Tam03}), $\pi_1(U)$ (and even its tame quotient $\pi_1^t(U)$) determines $U$ up to isomorphism as a scheme, if $U$ is an affine dense open subset of the projective line over $\overline{\mathbb F}_p$.  In addition, while the philosophy that motivated Abhyankar's Conjecture \ref{dAbhConj} for affine curves also suggests that the same should hold for complements of hyperplanes in general position in $\mathbb P^n$, that turned out not to be the case (see Section~\ref{dShighdim} below).  Nevertheless, the conjecture for affine curves over an algebraically closed field did turn out to be true in general.

\subsection{Examples}
Let $k$ be algebraically closed of characteristic $p$.  

\subsubsection{Basic examples}
It is easy to see explicitly that there are infinitely many nonisomorphic \'{e}tale $\ints/p$-covers of $\aff^1_k$, as already alluded to in the introduction.  For example, there are the covers given 
by the equations
$$y^p - y = x^{-n},$$
where $x$ is a coordinate on $\aff^1_k$ and $n$ ranges through $\nats - p \nats$. 
Taking fiber products gives infinitely many \'{e}tale $(\ints/p)^m$-covers of $\aff^1_k$ in characteristic $p$ for any $m$.  
For any \emph{abelian} $p$-group $G$, it is not difficult to extend this idea and to write down \'{e}tale $G$-covers of $\aff^1_k$ using
Artin-Schreier-Witt theory (see, e.g., \cite{La:al}).

\subsubsection{Abhyankar's examples}\label{aSAbhEx}

Abhyankar exhibited many examples of $G$-covers of $\aff^1_k$ in characteristic $p$ for non-abelian quasi-$p$ groups $G$.  Much of this work is 
collected in \cite{Abh92}.  For instance, let $x$ be a coordinate on $\aff^1_k$, and let $f: Y \to \aff^1_k$ be given by the equation
\begin{equation}\label{aEgenab}
y^n - ax^sy^t + 1 = 0,
\end{equation}
where $a \in k$ is not zero, where $s$, $t$, and $n$ are positive integers such that $t \not \equiv 0 \pmod{p}$, and where $n = p + t$.  One can
readily verify that $f$ is indeed an \'{e}tale map, and the Galois closure  $g: Z \to \aff^1_k$ of $f$ is 
an \'{e}tale $G$-cover of $\aff^1_k$.
Abhyankar (\cite[Section 11]{Abh92}) calculated the groups $G$ for various values of $n$, $s$, and $t$.
\begin{thm}[Abhyankar]\label{aTabex}
The Galois group of the cover $Z \to \aff^1_k$ given by taking the Galois closure of the cover of $\aff^1$ given by (\ref{aEgenab}) is 
$$ \begin{cases}
{\rm PSL}_2(p) = {\rm PSL}_2(n-1) & \text{ if } t = 1 \\
{\rm PSL}_2(8) = {\rm PSL}_2(n-1) & \text{ if } t = 2 \text{ and } p = 7 \\
S_n & \text{ if } p = 2 \\
A_n & \text{otherwise}
\end{cases}
$$
\end{thm}
Furthermore, Abhyankar gave other examples of \'{e}tale $S_n$ and $A_n$-covers of $\aff^1_k$ using the equation
$$y^n - ay^t + x^s = 0,$$ with $a$, $s$, $t$, and $n$ as above.  Abhyankar also gave many explicit
examples of covers of $\aff^1_k$ for various linear algebraic groups over $\F_q$,
with $q$ a power of $p$.  See Section \ref{SArithmeticalConj} for more details. 

\subsubsection{A result of Nori}
We mention here a result of Nori (see, e.g., \cite[Section 3.2]{Ser92})
that, although not used in the proof of Abhyankar's conjecture, gave more evidence for its truth.

\begin{theorem}\label{aTnori}
Let $\Sigma$ be any semisimple, simply connected algebraic group over $\F_q$, where $q$ is a power of $p$.  Then
$\Sigma(\F_q) \in \pi_A(\aff^1_k)$.
\end{theorem}

\proof[Sketch of proof]
Consider the Lang isogeny $L: x \mapsto x^{-1} F(x)$ on $\Sigma$, where $F$ is the Frobenius map.  By looking at the fiber above the identity,
one sees that this is a $\Sigma(\F_q)$-cover. 
If $U^+$ and $U^-$ are the unipotent radicals of opposite Borel subgroups of $\Sigma$, then
the map
$$\iota: U^+ \times U^- \to \Sigma$$ is a closed immersion.  Pulling back $L$ by $\iota$ gives a $\Sigma(\F_q)$-cover of $U^+ \times U^-$ (one can
show it remains connected).  Now, $U^+ \times U^-$ is isomorphic to affine space.  Then one shows, using a version of Bertini's theorem, that
there exists an affine line $\ell$ in $U^+ \times U^-$ such that the cover remains connected when restricted above $\ell$.  This is the
$\Sigma(\F_q)$-cover we seek.
\qed

\subsection{Proof of Abhyankar's conjecture}\label{aSabproof}

Assume throughout Section \ref{aSabproof} that $k$ is an algebraically closed field of characteristic $p>0$.
The proof of Abhyankar's Conjecture for affine curves (Conjecture~\ref{dAbhConj}) involves separate contributions from Serre, Harbater, and Raynaud.  Serre proved results about the cohomological dimension of curves, and his contribution proved Conjecture \ref{dAbhConj} for solvable group covers of the affine line.
Harbater developed formal patching methods which allowed for an inductive approach for realizing Galois groups of 
affine covers.  Raynaud, using semi-stable reduction of curves as well as another form of patching, proved Conjecture \ref{dAbhConj} 
for the affine line.  This, together with formal patching, proved Conjecture \ref{dAbhConj} for all affine curves.

\subsubsection{The case $U = \aff^1_k$}
Let $G$ be a quasi-$p$ group.  To show that $G \in \pi_A(\aff^1_k)$, one uses induction on $|G|$, the case $|G| = 1$ being 
trivial.  The first major breakthrough was due to Serre.

\begin{theorem}[Serre, \cite{Ser90}]\label{aTserre}
Let 
\begin{equation}\label{aEserre}
1 \to H \to G \to G/H \to 1
\end{equation}
be an exact sequence of finite groups, with $H$ solvable and $G$ quasi-$p$.  If $G/H \in \pi_A(\aff^1_k)$, then $G \in \pi_A(\aff^1_k)$.
\end{theorem}

\begin{cor}
Abhyankar's conjecture for the affine line is true for solvable groups, in particular, for $p$-groups.
\end{cor}

\proof[Sketch of proof of Theorem \ref{aTserre}]  
By induction, one reduces to the case where $H$ is elementary abelian of exponent $\ell$, for a prime $\ell$, with the conjugation action of $G/H$
on $H$ irreducible.  Now, one can show, using \'{e}tale cohomology, that $\pi_1(\aff^1_k)$ has cohomological dimension $1$, and is thus projective.
In particular, a surjection $\ol{\varphi}: \pi_1(\aff^1_k) \to G/H$ lifts to a morphism $\varphi: \pi_1(\aff^1_k) \to G$.  
Now, if (\ref{aEserre}) is not split, then $\varphi$ is surjective, and thus $G \in \pi_A(\aff^1_k)$.  
If (\ref{aEserre}) is split, then $\varphi$ might not be surjective.  However, a group cohomology argument shows that there will exist a choice of surjective
lift $\varphi$, provided that the natural inclusion 
\begin{equation}\label{aEinclusion}
H^1(G/H, H) \subseteq H^1(\pi_1(\aff^1_k), H)
\end{equation} induced by $\ol{\varphi}: \pi_1(\aff^1_k) \to G/H$ is strict, where $G/H$ acts on $H$ via conjugation in 
$G$.

The proof then breaks into two cases, depending on whether or not $\ell = p$.  If $\ell = p$, then one shows that
$H^1(\pi_1(\aff^1_k), H)$ is infinite, and thus the inclusion in (\ref{aEinclusion}) must be strict.  If $\ell \neq p$, then 
the inclusion in (\ref{aEinclusion}) might not be strict as is, but we can make it strict by replacing $\ol{\varphi}$ by some $\ol{\varphi}'$.
Namely, $\ol{\varphi}$ corresponds to an \'{e}tale $G/H$-cover of $\aff^1_k$.  We can pull this cover back by the $m$th power map on
$\aff^1_k$, where $m$ is any natural number prime to $p$.  This gives a new \'{e}tale $G/H$-cover of $\aff^1_k$, corresponding
to some $\ol{\varphi}': \pi_1(\aff^1_k) \to G/H$.  One then shows that $$d' = md,$$ 
where $d$ (resp.\ $d'$) is the dimension of $H^1(\pi_1(\aff^1_k), H)$ as an $\F_{\ell}$-vector space, where the action on $G/H$ is 
via the map $\ol{\varphi}$ (resp.\ $\ol{\varphi}'$).
Thus we get a strict inclusion in (\ref{aEinclusion}) whenever $m > 1$.
\qed

\medskip

Given Theorem \ref{aTserre}, we may assume that $G$ has no nontrivial normal $p$-subgroup.  Raynaud's proof of Abhyankar's conjecture for $\A^1_k$
(\cite{Ray94}) is subdivided into two cases, addressed in Theorems \ref{aTraynaud1} and \ref{aTraynaud2} below.

\begin{theorem}[Raynaud]\label{aTraynaud1}
Fix a $p$-Sylow subgroup $S$ of $G$, and let $G(S) \subseteq G$ be generated by all proper quasi-$p$ subgroups
of $G$ with a $p$-Sylow group contained in $S$. 
 If $G(S) = G$, then $G \in \pi_A(\aff^1_k)$.
\end{theorem}

\begin{remark}
Since $G(S)$ only depends on the particular $p$-Sylow subgroup $S$ up to conjugacy, the hypothesis in Theorem \ref{aTraynaud1} depends
only on $G$.
\end{remark}

\begin{remark}
In this case, the assumption that $G$ has no nontrivial normal $p$-subgroup is unnecessary.
\end{remark}

\proof[Sketch of proof of Theorem \ref{aTraynaud1}]
By induction, each proper subgroup $G_i$ of $G$ whose $p$-Sylow subgroup is contained in $S$ occurs as the Galois group of an \'{e}tale cover
$f_i: Y_i \to \aff^1_k$. By Lemma \ref{1SRH} (1), the inertia groups of $f_i$ at $\infty$
are isomorphic to some $P_i \rtimes \ints/m_i$, where
$P_i$ is a $p$-group and $p \nmid m_i$.  One then takes the pullback of $f_i$ via the map $\aff^1_k \to \aff^1_k$ given by raising to the $m_i$th power.
Since $p \nmid m_i$, Abhyankar's Lemma \ref{Lablemma} shows that this pullback is an \'{e}tale $G_i$-cover with inertia groups
$P_i$.  Thus we may assume that the inertia groups of $f_i$ are $p$-groups (which, by assumption, are contained in $S$).

In this case, Raynaud shows, using the technique of rigid patching, that there is an \'{e}tale $G$-cover $f: Y \to 
\aff^1_k$ with inertia groups at $\infty$
isomorphic to $S$.  In order to use rigid patching, Raynaud must work over the field $K := k((t))$, rather than $k$.  
Once the desired $G$-cover is constructed over $K$, an appropriate specialization of $t$ gives a $G$-cover defined over $k$.  For an introduction
to rigid patching, see, e.g., \cite{Ha:pg}, Section~4. \qed

\begin{theorem}[Raynaud]\label{aTraynaud2}
In the situation of Theorem \ref{aTraynaud1}, assume that $G(S) \neq G$, and that $G$ has no nontrivial normal $p$-subgroup. 
Then $G \in \pi_A(\aff^1_k)$.
\end{theorem}

\proof[Sketch of proof]
The field $K_0 := \Frac(W(k))$ is a complete discretely valued field in characteristic zero with residue field $k$, where $W(k)$ is the ring of Witt vectors over $k$ 
(see \cite{Serr}, Chapter~II, Section~6).  Let $\ol{K}$ be its algebraic closure.
Since $G$ is quasi-$p$, it can be generated by, say, $r$ elements of order a power of $p$.  
Since $\ol{K}$, like $\complex$, is an algebraically closed field of characteristic zero, the Lefschetz principle shows that there is a 
branched $G$-cover $f_{\ol{K}}: Y_{\ol{K}} \to \proj^1_{\ol{K}}$ which is branched above $r+1$ points, with $p$-power branching indices 
(just as if we were working over $\complex$).  For non-trivial $G$, by taking $r$ large enough we can force the genus of $Y$ to be at least $2$. 
Furthermore, we may assume that the branch points lie in $\aff^1_{\ol{K}}$, and that they are 
pairwise non-congruent modulo the maximal ideal of the ring of integers of 
$\ol{K}$.  This cover descends to a cover $f_K : Y_K \to \proj^1_K$ 
over some finite extension $K/K_0$.

After possibly replacing $K$ with a further finite extension, one can show that $f_K$ has a \emph{stable model} over the valuation ring $R$ of $K$.
That is, there exists a relative $R$-curve $Y_R$ that has stable special fiber (i.e., only nodes for singularities) 
and for which there exists an $R$-linear $G$-action on $Y_R$ such that the quotient morphism $f_R: Y_R \to P_R := Y_R/G$ has generic fiber $f_K$.
The special fiber $P_k$ of $P_R$ will have a tree of $\proj^1_k$'s for irreducible components (and this tree will consist of more than one component).
In \cite[Section 6]{Ray94}, a detailed analysis of the properties of stable models of 
Galois covers of curves shows that, under the assumptions in Theorem \ref{aTraynaud2}, there is a ``tail" irreducible component, say, 
$D_k = \proj^1_k$, for which $f_R \times_{P_R} D_k$ is a $G$-cover of $\proj^1_k$ branched at one point.  This gives us the \'{e}tale $G$-cover of
$D_k$ we seek. \qed

\begin{remark}
It is interesting and perhaps unexpected that the hypothesis $G(S) \neq G$ is \emph{absolutely necessary} for the proof of Theorem \ref{aTraynaud2},
whereas the opposite is necessary for the proof of Theorem \ref{aTraynaud1}!
\end{remark}

\subsubsection{Patching and the general case}
Harbater's proof of the general case of Abhyankar's Conjecture (\cite{Har94}, simplified in \cite{Har03}) is divided into two steps.   
The first step is to prove the conjecture for the case $\aff^1_k - \{0\}$, and the second is to use that to obtain the general case.
\begin{theorem}[Harbater]\label{aTharbater}
If $G$ is a finite group such that $G/p(G)$ is cyclic of prime-to-$p$ order, then $G \in \pi_A(\aff^1_k - \{0\})$.
\end{theorem}

\proof[Sketch of proof] 
We construct an \'{e}tale $G$-cover over $\aff^1_K - \{0\}$, where $K = k((t))$.  An appropriate specialization of $t$ then yields the desired 
$G$-cover of $\aff^1_k - \{0\}$.

First, a group-theoretic simplification allows us to assume that $G \cong p(G) \rtimes \ol{G}$, where $\ol{G}$ is cyclic of prime-to-$p$ order.  Then, one
proves a strengthening of Theorem \ref{aTserre}, showing that, if $H$ is a $p$-group, then not only does Theorem \ref{aTserre} 
apply to $\aff^1_K - \{0\}$ as well, but also that one can specify the local behavior of this cover over the complete local ring at a point (in fact, finitely 
many points).  Since a $\ol{G}$-cover of $\aff^1_K - \{0\}$ is easy to construct (use the equation $y^{|\ol{G}|} = x$), the above strengthening yields
a $P \rtimes \ol{G}$-cover $g: V \to \aff^1_K - \{0\}$, where $P$ is a $p$-Sylow group of $G$; and this can be chosen with specified local behavior.  

On the other hand, the $\aff^1_k$-case of Abhyankar's conjecture can be used as a black box to build an \'{e}tale $p(G)$-cover $h: W \to \aff^1_k$.
In fact, we can force the inertia groups at $\infty$ to be isomorphic to $P$, using Abhyankar's Lemma and deformations of covers.  Taking the smooth projective completion of $h$, and 
looking at the (fraction fields of the) complete local ring at a ramification point with inertia group $P$, one obtains a $P$-Galois extension $L/K$.  By choosing $g$ above so that its restriction
above some $K$-point of $\aff^1_K - \{0\}$ is the map
$$\coprod_{i=1}^{|\ol{G}|} \Spec L \to \Spec K,$$ we are able to use ``formal patching" (see \cite{Ha:pg}) to glue $g$ and a thickening of $h$ together, 
giving an \'{e}tale $G$-cover of $\aff^1_K - \{0\}$.  The cover is connected because $G$ is generated by $p(G)$ and $P \rtimes \ol{G}$. \qed

\begin{remark}
In fact, the cover constructed in the proof above will be \emph{tamely} ramified above $0$.
\end{remark}

\proof[Sketch of the proof of the general case of Abhyankar's conjecture] 
Consider the general case, where $U$ is a curve of type $(g, r)$, with $r > 0$, and $G$ is a group such that $G/p(G)$ can be generated by $2g + r - 1$
elements.   As in Theorem
\ref{aTharbater}, one can reduce to the case where $G = p(G) \rtimes \ol{G}$, where $p \nmid |\ol{G}|$ and $\ol{G}$ is generated by
$2g + r -1$ elements.  By \cite[XIII, Corollary 2.12]{SGA1}, there is a 
$\ol{G}$-cover $f: Y \to U$.  Let $C$ be an inertia group of $f$ at some ramification point $\xi$; this is a prime-to-$p$ cyclic group.   
Let $H$ be the subgroup of $G$ generated by $p(G)$ and $C$; using group theory we may reduce to the case that this is a semi-direct product.  By Theorem~\ref{aTharbater} and the remark afterwards, there is an $H$-Galois cover $h: W \to \aff^1_K - \{0\}$, tamely pramified at the point $0$, with inertia group $C$ there.  One then patches the covers $f$ and $h$.  This is done so that the tame ramification of $h$ above $0$  ``cancels out" the
tame ramification of $f$ at $\xi$, thus leaving a total of $r$ branch points of the (projective completion of the) final cover.  
Since $H$ and $\ol{G}$ generate $G$, patching these together yields a $G$-cover of type $(g, r)$ curve over
$K$.  Again, an appropriate specialization yields the desired cover over $k$, showing that $G \in \pi_A(U)$. \qed

\begin{remark}
As in the previous remark, the cover constructed in the proof above will be \emph{tamely} ramified above all but one point of $X$, where $X$ is the smooth projective 
completion of $U$.
\end{remark}

A related proof of the general case of Conjecture~\ref{dAbhConj}, also relying on the case of the affine line, can be found in \cite{P95}; see the corollary to Theorem~B there.

\section{Abhyankar's Inertia Conjecture} \label{SConj2}

Let $k$ be an algebraically closed field of characteristic $p > 0$.
Earlier, we discussed the fact that the prime-to-$p$ fundamental group of the affine line ${\mathbb A}^1_k$ is trivial, which implies that
only quasi-$p$ groups can occur as Galois groups of covers $\phi:Y \rightarrow \PP^1$
branched only at $\infty$.  As we have seen, Abhyankar's Conjecture \ref{dAbhConj} implies that all quasi-$p$ groups do, in fact, occur.

\subsection{Abhyankar's Inertia Conjecture} \label{SSConj2}

Now, consider the fact that the tame fundamental group of the affine line ${\mathbb A}^1_k$ is trivial.
This places restrictions on the subgroups of $G$ which can occur as inertia groups of a $G$-Galois cover $\phi:Y \to \PP^1$ branched only at infinity.
The inertia group $G_0$ at a ramified point of $\phi$
is a semi-direct product of the form $G_1 \rtimes \Z/m$ where $G_1$ is a $p$-group and $p \nmid m$ (Lemma \ref{1SRH} (1)).  
The subgroup $J \subset G$ generated by the conjugates of $G_1$ is normal.
Since the $G/J$-Galois quotient cover is a tame Galois cover of $\PP^1$ branched only at $\infty$, it must be trivial; thus $J=G$.
Based on this, Abhyankar stated the currently unproven Inertia Conjecture.

\begin{conj}[Abhyankar's Inertia Conjecture, \cite{Abh01}, Section 16]\label{Iconjecture}
Let $G$ be a finite quasi-$p$ group.  Let $G_0$ be a subgroup of $G$
which is an extension of a cyclic group of order prime-to-$p$ by a $p$-group $G_1$.  
Then $G_0$ occurs as the inertia group of a ramified point of a $G$-Galois cover $\phi:Y \to \PP^1$
branched only at $\infty$ if and only if the conjugates of $G_1$ generate $G$.
\end{conj}

There are not many classes of groups for which Abhyankar's Inertia Conjecture has been verified.  
Abhyankar's Inertia Conjecture is trivially true when $G$ is a $p$-group because a $p$-group cover of the affine line must be totally ramified over infinity.
In this section, we first reformulate the inertia conjecture into two distinct parts and discuss evidence
for it.  Then we discuss potential constraints on the inertia group, including a new result relating the structure of the inertia group 
with a congruence condition on the genus of a subquotient. 
 
\subsection{Reformulation of Abhyankar's Inertia Conjecture}

In essence, Abhyankar's Inertia conjecture consists of two separate claims: one about the possibilities for the Sylow $p$-subgroup of the inertia group and one about possibilities for the tame part of the inertia group.

\begin{conj}[Inertia Conjecture, \cite{Abh01}, Section 16]\label{Iconjecture2}
Let $G$ be a finite quasi-$p$ group.
\begin{enumerate}
\item{Wild part:} 
A $p$-group $G_1 \subset G$ occurs as an inertia group for a $G$-Galois cover $\phi:Y \to \PP^1$
branched only at $\infty$ if and only if the conjugates of $G_1$ generate $G$.
\item{Tame part:} 
Suppose a $p$-group $G_1 \subset G$ occurs as an inertia group for a $G$-Galois cover $\phi:Y \to \PP^1$ branched only at $\infty$.
Suppose $\beta \in G$ has prime-to-$p$ order and is contained in the normalizer of $G_1$ in $G$.
Then the semi-direct product $G_0=G_1 \rtimes \langle \beta \rangle$
occurs as an inertia group for a $G$-Galois cover $\phi:Y \to \PP^1$ branched only at $\infty$.
\end{enumerate}
\end{conj}

\subsection{Changing the inertia group}

Abhyankar's Lemma \ref{Lablemma} allows one to decrease the order of the tame part of an inertia group by pulling back by a tame cover.
Because of Abhyankar's Lemma, it is valuable to search for inertia groups with large tame order.  Here is a restatement of the lemma for Galois covers of the affine line.

\begin{lemma}[Abhyankar's Lemma]
Suppose $\phi:Y \to \PP^1_k$ is a $G$-Galois cover branched only over $\infty$ with inertia group $I$ of the form $P \rtimes \Z/m$.
Let $\psi: \PP^1_k \to \PP^1_k$ be a $\Z/n$-Galois cover branched at $0$ and $\infty$.
Then the pullback $\psi^* \phi$ of $\phi$ by $\psi$
is a $G$-Galois cover branched only over $\infty$ with inertia group of the form $P \rtimes \Z/(m/{\rm gcd}(m,n))$.
\end{lemma}

The main results about changing inertia groups are found in \cite[Theorem 2]{Harb93} and \cite[Theorem 3.6]{Har03}.  
These are long results to state in full generality, but here is a special case.
Suppose $\phi:Y \to X$ is a $G$-Galois cover of curves having inertia group $I$ of the form $P \rtimes \Z/m$ above a branch point $b$.  
Suppose $I \subset I' \subset G$ and the index $[I':I]$ is a power of $p$.
Then it is possible to modify the cover $\phi$ to another $G$-Galois cover $\phi':Y' \to X$ with the same branch locus as $\phi$, whose ramification at branch points other than $b$ is the 
same as for $\phi$ and whose inertia groups above $b$ are isomorphic to $I'$.
 
In particular, choosing $\phi$ to be a Galois cover of the affine line, one deduces the following theorem.
Because of this theorem, it is valuable to search for inertia groups whose Sylow $p$-subgroup has smaller order than the Sylow $p$-subgroup of $G$.

\begin{theorem}[\cite{HaICM}, Section 4.1] \label{Tsylow} 
Suppose $G$ is a quasi-$p$ group.  Then there exists a $G$-Galois cover $\phi: Y \to \PP^1_k$ branched only over $\infty$ whose inertia groups are the Sylow $p$-subgroups of $G$.
\end{theorem}

\subsection{Realizing inertia groups via reduction} \label{Sinertiareduction}

As we saw in the proof of Theorem \ref{aTraynaud2}, one strategy for realizing covers of the affine line with given inertia group is to
study the reduction of $G$-Galois covers of the projective line in characteristic $0$.
If $p$ divides one of the ramification indices, then the cover may have bad reduction to characteristic $p$.
Under some conditions on the reduction and on the Galois group, there is a component of the stable
reduction which is a (connected) $G$-Galois cover of $\PP^1_k$ wildly ramified above only one point.

The following two results are based on the reduction method:

\begin{theorem}[\cite{BP}, Theorem 2] \label{TinerRRR} 
The inertia conjecture is true for the alternating group $G=A_p$ and the projective special linear group $G={\rm PSL}_2(p)$ for $p \geq 5$.
\end{theorem}

\begin{theorem}[\cite{obusabhy}, Corollary 3.6]  \label{AndrewPSL} 
Let $p \geq 7$.
Suppose $L$ is a prime such that $G \simeq {\rm PSL}_2(L)$ has order divisible by $p$.
Suppose $I$ is cyclic of order $p^r$ or dihedral of order $2p^r$ with $1 \leq r \leq v_p(|G|)$.
Then there exists a $G$-Galois cover
$\phi: Y \to \PP^1_k$ branched at only one point with inertia groups isomorphic to $I$.
\end{theorem}

\begin{remark}
It appears that Theorem \ref{AndrewPSL} is the first existence result for covers of the affine line whose inertia groups are strictly contained in a Sylow $p$-subgroup.
Other examples of this can be produced using the fiber product construction in \cite[Section 4]{kumar}.  
\end{remark}

\begin{remark}
For fixed $p$, the condition that $|{\rm PSL}_2(L)|$ is divisible by $p$ is satisfied by infinitely many primes $L$.
\end{remark}

\begin{remark}
It is usually difficult to obtain complete information about the stable reduction.
Typically, there is a vanishing cycle formula which gives information about the sum of the conductors for the covers of the terminal components in the stable reduction.
Sometimes it is possible to use extra information to guarantee that one of the covers has non-trivial tame inertia.
The reduction of a cover with small field of definition is more likely to produce such a cover.
The authors of \cite{BP} use additional information about the reduction of the modular curve $X(2p)$ when $G={\rm PSL}_2(p)$
and the existence of subgroups with index $p$ when $G=A_p$ to produce non-trivial tame inertia.

Also, the stable reduction method produces \emph{connected} $G$-Galois covers of the affine line only under some 
group theoretic conditions: that $G$ has no normal $p$-subgroup and that $G(S) \not = G$. 
\end{remark}

\subsection{Realizing inertia groups via explicit equations} \label{Sinertiaexplicit}

Abhyankar's Inertia Conjecture has in some cases been realized using Abhyankar's original equations.

\begin{theorem}[\cite{MP12}, Theorem 1.2] \label{TinerMP} 
If $p \equiv 2 \bmod 3$ is an odd prime, then the inertia conjecture is true for the alternating group $A_{p+2}$.
\end{theorem}

More generally, the authors consider the group $G=A_{p+s}$ with $2 \leq s <p$.
They use Abhyankar's explicit equation for an extension of $k(x)$ with Galois closure $A_{p+s}$:
\[g_s(y)=y^{p+s}-xy^s+1\in k(x)[y].\] 
From this explicit equation, the authors calculate a Newton polygon whose slopes determine the conductor of the cover.  
The method is in some way akin to 
Abhyankar's method of throwing away roots (\cite[Section 4]{Abh92}).
When $s=2$, there is a maximal subgroup of the form $\Z/p \rtimes \Z/m$ containing a fixed $p$-cycle and it is realized as the
inertia group of the Galois closure of this extension when $p \equiv 2 \bmod 3$ is odd. 

The authors also reprove Abhyankar's Inertia Conjecture for $A_p$ in this method, 
using Abhyankar's explicit equation for an extension of $k(x)$ with Galois closure $A_p$: 
\[f(y)=y^p-xy^{p-t}+x \in k(x)[y].\]

In addition, the proofs of Theorems \ref{TinerRRR}, \ref{AndrewPSL}, and \ref{TinerMP} determine information about the ramification filtrations which can be realized for covers with the given Galois cover; see Section \ref{Srealizingcond}.

\subsection{Higher Ramification Groups}\label{RamGrps}

The proofs of the results in Sections \ref{Sinertiareduction} and \ref{Sinertiaexplicit} 
use information about higher ramification groups, as defined in Section \ref{Shigherram}.
Suppose that $\phi:Y \to X$ is a $G$-Galois cover wildly ramified at some point $\eta$.
Here is more information about the ramification filtration in the special case that
the order of the inertia group $G_0$ is strictly divisible by $p$. 
It follows from Lemma \ref{1SRH} that $G_0$ is a semidirect product of the form $\Z/p \rtimes \Z/m$
for some prime-to-$p$ integer $m$.  

By Lemma \ref{Dlowerjump}, 
the {\it lower jump} of $\phi$ of $\eta$ is the largest integer $J$
such that $G_J\not =\{1\}$.
Recall that the lower numbering on the filtration from Definition \ref{Dlowerjump}
is invariant under sub-extensions and that 
there is a different indexing system on the filtration, whose virtue
is that it is invariant under quotient extensions.

By Definition \ref{Dherbrand}, $G^{\varphi(i)}=G_i$ where $\varphi(i)=|G_0|^{-1}\sum_{j=1}^i|G_j|$.  
In this case, $\varphi(J)=J/m$.  The rational number $\sigma=J/m$ is the
{\it upper jump}; it is the jump in the filtration of the higher
ramification groups in the upper numbering.

\begin{defn} \label{2Dalpha}
Let $G_0 \simeq \Z/p \rtimes \Z/m$ with chosen generators $\tau$ of order $p$ and $\beta$ 
of order $m$.  
Let $g$ be the homomorphism $g: \langle \beta \rangle \to \Aut(\langle \tau \rangle)$
given by $g(\beta):\tau \mapsto \beta \tau \beta^{-1}$.  
Define $m''=|{\rm Ker}(g)|$ and $m'=m/m''$.
\end{defn}

One says that $m''$ is the {\it central} part of the tame ramification since
$m''$ is the order of the prime-to-$p$ part of the center of $G_0$.
Also $\langle \beta^{m'} \rangle$ is precisely the subgroup of order $m''$ which 
acts trivially on $\Z/p$.  
One says that $m'$ is the {\it faithful} part of the tame ramification since the image of 
$g$ has order $m'$ in ${\rm Aut}(\langle \tau \rangle)$.  Note that $m'|(p-1)$.

Using ramification theory from \cite[IV, Proposition 9]{Serr} or field theory as in \cite[Lemma 1.4.1(iv)]{Pr:fam},
one can see that there is a congruence condition on the lower jump modulo $m$, namely \[{\rm gcd}(J, m)=m''.\]
In other words, $J/m$ has denominator $m'$ when written in lowest terms.

\subsection{An unexpected constraint on ramification} \label{Sbouwpriesconstraint}

There are necessary conditions on the ramification filtration arising from class field theory.
These conditions can be called local conditions in that they arise from conditions on $I$-Galois extensions of complete local rings.

However, it appears that there are some non-local conditions on the ramification filtration as well.
It turns out that the set of conductors at a ramification point with inertia group $I$ depends not only on $I$ but also 
in a subtle way on the group $G$.
In particular, there is a difference between the conductors that occur for $A_p$ and ${\rm PSL}_2(p)$ Galois covers despite the fact that the same inertia groups occur.

\begin{theorem}[\cite{BP}, Theorem 3] \label{BPyesno} 
Suppose that $p\geq 7$. Let $m=(p-1)/2.$ There exists an $A_p$-Galois
cover of $\PP^1_k$ branched only at $\infty$ with upper jump $\sigma=(p-2)/m$, but
there is no ${\rm PSL}_2(p)$-Galois cover of $\PP^1_k$ branched only at $\infty$ 
with upper jump $\sigma=(p-2)/m$.
\end{theorem}

The proof of this result uses lifting of semi-stable covers of curves to characteristic $0$.
The underlying reason behind this result is that the rigid triples
(e.g., \cite[Definition 2.15]{Vol96}) for these groups in characteristic $0$ are different;  
namely ${\rm PSL}_2(p)$ has many rigid triples whereas $A_p$ has very few. 
This idea is generalized further in \cite[Sections 3.2 and 3.3]{BP}.

\subsection{A potential new constraint on the inertia group} \label{Sconstraint}

This section states a new result relating the structure of the inertia group to the upper jump and genus of quotient extensions.
Before stating it, here is some elementary group theory.

\subsubsection{Group theory}

Let $G$ be a transitive subgroup of $S_n$.
Let $p$ be a prime such that $p \leq n < 2p$.  Then, $p^2 \nmid |G|$.
The following lemmas give an upper bound for the size of the inertia group. 

\begin{lemma}\label{NormalizerSp}
Let $\tau=(1,2,\dots, p)$. Then there exists $\theta \in S_p$ with order $p-1$ such that
$N_{S_p}(\langle\tau\rangle)=\langle\tau\rangle\rtimes\langle \theta \rangle$ and such that $\theta$ fixes $1$.
\end{lemma}

\begin{proof}
Let $n_p$ be the number of Sylow $p$-subgroups of $S_p$; then
$n_p=[S_p:N_{S_p}(\langle\tau\rangle)]$.
There are $(p-1)!$ different $p$-cycles in $S_p$, each generating 
a group with $p-1$ non-trivial elements.  It follows that
$n_p=(p-2)!$. Therefore, $|N_{S_p}(\langle\tau\rangle)|=p(p-1)$.

Clearly, $\langle\tau\rangle \subset N_{S_p}(\langle\tau\rangle)$.  Let $a$ be a generator of $\F_p^*$.  There exists
$\theta \in S_p$ such that $\theta \tau \theta^{-1}=\tau^a$ since all $p$-cycles in $S_p$ are conjugate.
Then $\theta \in N_{S_p}(\langle\tau\rangle)$.
Also $\theta$ has order $p-1$ since, for any $r$,
\[
\theta^{r}\tau\theta^{-r}=\tau^{a^{r}}.
\]
Conjugating $\theta$ by a power of $\tau$ gives a new choice of $\theta$ that fixes $1$.
\end{proof}

Recall that
$C_{S_n}(\langle\tau\rangle)=\langle\tau\rangle\times H \textrm{
where } H=\{\omega\in S_n:\omega \textrm{ is disjoint from }\tau\}$.

\begin{lemma}\label{NormalizerSn}
Let $G$ be a transitive subgroup of $S_n$ where $n=p+s$ and $2 \leq s <p$.
Let $\tau=(1,2,\dots, p)$ and let $\theta \in S_p$ satisfy the
properties of Lemma \ref{NormalizerSp}.  
Let $H_s \subset S_{p+s}$ be the subgroup of permutations of the set $\{p+1,p+2,\dots,p+s\}$.  
Then
$N_{G}(\langle\tau\rangle)$ is the intersection of $G$ with 
$(\langle\tau\rangle \rtimes \langle \theta \rangle) \times H_s$.
\end{lemma}

\begin{proof}
The permutation $\theta$ in the proof of Lemma \ref{NormalizerSp} has order $p-1$ and normalizes $\tau$.  
The elements of $H_s$ commute with $\tau$ and $\theta$.
Thus $G \cap (\langle\tau\rangle \rtimes \langle \theta \rangle) \times H_s \subset N_{S_{p+s}}(\langle \tau \rangle)$.
Performing a similar count as for Lemma \ref{NormalizerSp}, 
the number of Sylow $p$-subgroups in $S_{p+s}$ is $(p+s)!/(s!p(p-1))$, which must equal the 
index of $N_{S_{p+s}}(\langle\tau\rangle)$ in $S_{p+s}$.
Therefore $|N_{S_{p+s}}(\langle\tau\rangle)|=s!p(p-1)$. 
Thus $G \cap (\langle\tau\rangle \rtimes \langle \theta \rangle) \times H_s = N_{S_{p+s}}(\langle \tau \rangle)$.
The result follows by taking the intersection with $G$.
\end{proof}

\begin{notation} \label{Ncycletype}
Let $p \leq n < 2p$.
Suppose $G$ is a transitive subgroup of $S_n$ and is a quasi-$p$ group.
Without loss of generality, suppose $\tau=(1,2,\ldots, p) \in G$.
Suppose $I \subset G$ is a semi-direct product of the form
$\langle \tau \rangle \rtimes \langle \beta \rangle$ for some $\beta \in G$
having prime-to-$p$ order $m$.
Suppose the action of $\beta$ on  $\{p+1, \ldots, n\}$ breaks into $T$ disjoint cycles with lengths $n_1, \ldots, n_T$.
Then $\{n_1, \ldots, n_T\}$ is a partition of $n-p$.
Let $L={\rm lcm}(n_1, \ldots, n_T)$.
\end{notation}

The ideas in the following result were used frequently by Abhyankar.

\begin{prop} \label{Pramfibre}
With notation as in \ref{Ncycletype}, 
Suppose $\phi: Y \to \PP^1$ is a $G$-Galois cover branched only at $\infty$ 
with inertia $I$ at a ramification point $\eta$.
Consider the subgroup $S_{n-1} \subset S_n$ of elements fixing $1$.
Let $X_1$ be the quotient of $Y$ by $G \cap S_{n-1}$.
Then the fibre of $\psi: X_1 \to \PP^1$ above $\infty$ consists of $T+1$ points, with inertia groups of 
order $p, n_1, \ldots, n_T$.
\end{prop}

\begin{proof}
Let $H = G \cap S_{n-1}$, where $S_{n-1}$ is the subgroup fixing $1$. 
The symmetric group acts on $\{1, ..., n\}$ via a right action and the 
set of cosets $H \backslash G$ is in bijection with $\{1, ..., n\}$, via the image of $1$.
Let $\psi: X_1 \to \PP^1$ be the quotient of $Y$ by $H$.
Then $\psi$ has degree $n$ and the fibers of $\psi$ can be indexed by the numbers $1, \ldots, n$.  

After labeling a point of $Y$ in $\phi^{-1}(\infty)$ to correspond to the identity coset, 
there is a bijection from $G/I$ to the fiber of $\phi$ above $\infty$.  
The action of $G$ on $\phi^{-1}(\infty)$ is given by left multiplication.

Now, the points of $X_1$ in $\psi^{-1}(\infty)$ correspond to the orbits of $H$ on $G/I$ (via left action).  
In other words, they correspond to double cosets $H \backslash G/I$, which correspond to orbits of $I$ acting on the right on $\{1, \ldots, n\}$.  Additionally, the lengths of the orbits correspond to the ramification indices.  
From Notation \ref{Ncycletype} and Lemma \ref{NormalizerSn},
there is one orbit of length $p$, and $T$ additional orbits of lengths $n_1, \ldots, n_T$.

\end{proof}

The following proposition is a new result which 
gives a relationship between the central and faithful parts of the tame
ramification of a Galois cover of the affine line.  It is inspired by \cite{bouwen}. It is not immediately apparent whether this 
relationship places new constraints on inertia groups of covers of the affine line since it depends on the genus of a quotient of the cover.

\begin{proposition}
Assume the context of Notation \ref{Ncycletype} and Proposition \ref{Pramfibre}, and let $g_1$ be the genus of $X_1$.  Then
the upper jump $\sigma$ (see Section~\ref{RamGrps})
satisfies the relationship:
\[\displaystyle \sigma = \frac{2g_1 + n + T-1}{p-1}.\]
In particular, the order of the faithful part of the tame ramification satisfies the relationship:
\[\displaystyle m'=\frac{p-1}{{\rm gcd}(p-1, 2g_1+n+T-1)}\]
and the order of the central part of the tame ramification satisfies the relationship:
\[\displaystyle m''=\frac{L}{{\rm gcd}(m',L)},\]
where $L={\rm lcm}(n_1, \ldots, n_T)$.
\end{proposition}

\begin{proof}
The strategy of the proof is to calculate the genus of $Y$ in two ways using the Riemann-Hurwitz formula.
First, by hypothesis, the $G$-Galois cover $\phi: Y \to \PP^1$ is branched only at $\infty$ 
with inertia $I$ at a ramification point.  Recall that $I \subset G$ is a semi-direct product of the form
$\Z/p \rtimes \Z/m$.  Let $J$ be the jump in the ramification filtration of $\phi$ in the lower numbering.  
By Theorem \ref{1SRH}(3), 
\begin{equation} \label{Ehurwitz1}
2g_Y-2=|G|(-2) + \frac{|G|}{|I|}(|I|-1 + (p-1) J).
\end{equation}

By Proposition \ref{Pramfibre}, 
the fibre of $\psi: X_1 \to \PP^1$ above $\infty$ consists of $T+1$ points, with inertia groups of 
order $p, n_1, \ldots, n_T$.  Denote these points by $P_0, P_1, \ldots, P_T$.

Consider the subcover $\phi_1:Y \to X_1$ which is Galois with degree $|G|/n$.
Above $P_0$, the cover $\phi_1$ is tamely ramified with ramification index $m$.
Above $P_i$ for $1 \leq i \leq T$, the cover $\phi_1$ is wildly ramified with ramification index $pm/n_i$.
By \cite[IV, Proposition 2]{Serr}, the lower jump of $\phi_1$ above $P_i$ is still $J$.
By Theorem \ref{1SRH}(2), 
\begin{equation} \label{Ehurwitz2}
2g_Y - 2 = \frac{|G|}{n} (2g_1 -2) + \frac{|G|}{nm}(m-1) + 
\sum_{i=1}^T \frac{|G|n_i}{npm}\bigl(\frac{pm}{n_i} -1 + (p-1)J\bigr).
\end{equation}

Setting the right hand sides of equations \eqref{Ehurwitz1} and \eqref{Ehurwitz2} equal and multiplying by 
$n/|G|$ yields
\begin{equation} \label{Ehurwitz3}
-2n + n-\frac{n}{pm}+\frac{n(p-1)}{pm}J = 2g_1 -2 + \frac{m-1}{m} + 
\sum_{i=1}^T (1 - \frac{n_i}{pm} + \frac{n_i(p-1)}{pm}J).
\end{equation}

Recall that $\sum_{i=1}^T n_i=n-p$.  Substituting this yields that:
\begin{equation} \label{Ehurwitz4}
-2n + n-\frac{n}{pm}+\frac{n(p-1)}{pm}J = 2g_1-2 + \frac{m-1}{m} +
T - \frac{n-p}{pm} + \frac{p-1}{pm}(n-p)J.
\end{equation}

Thus $\frac{J}{m}(p-1)= 2g_1 - 1 + T + n$ and the first claim follows since $\sigma=J/m$.
The second claim follows from the fact that $m'$ is the smallest positive integer such that $m' \sigma$ is integral.
The third claim follows from the fact that $m=L$ is the order of the tame inertia of $\phi$. 
\end{proof}

This proposition allows one to reinterpret the Inertia Conjecture in terms of the congruence class of 
$g_1$ modulo $p-1$.

\begin{example}
Suppose $G=A_{p+t}$ with $2 \leq t <p$.
Suppose the cycle type of $\beta$ on $\{p+1, \ldots, p+t\}$ is a $t$-cycle.
Then $\sigma = \frac{2g_1 + p+t}{p-1}$.
\end{example}

\subsection{Open problems}

Among the many quasi-$p$ groups that are worthy of investigation, the alternating groups $A_n$ with $p \leq n < 2p$ 
appear the most accessible.
In this case, to prove (the tame part of) Abhyankar's Conjecture, 
one must show that each subgroup $I$ of $N_{A_{n}}(\langle \tau \rangle)$ having the form $\Z/p \rtimes \Z/m$
occurs as the inertia group of an $A_n$-Galois cover of the affine line.  
These subgroups are partially ordered under inclusion; we say that $I$ is {\it a maximal inertia group for $A_n$}
if $I \not \subset I'$ for any $I' \subset N_{A_{n}}(\langle \tau \rangle)$ having the form $\Z/p \rtimes \Z/m$.
To prove that Abhyankar's Conjecture is true for the group $A_n$, it suffices to prove that each
of the maximal $I$-groups for $A_n$ occurs as the inertia group of an $A_n$-Galois cover of the affine line.
As $n$ increases, the number of maximal $I$-groups grows quickly.

The maximal inertia groups $I$ for $A_n$ are of two types.  The first type are those which stabilize an element, say $n$,
and thus are contained in $A_{n-1}$.  One could ask if an $A_{n-1}$ cover of the affine line with inertia $I$ could be used 
to produce an $A_n$ cover of the affine line with inertia $I$.  It appears difficult to do this, using either field theory or 
formal patching.  

The second type are those which do not stabilize an element.  Even for small $s=n-p$, most of these have not 
been constructed explicitly.
For example, the Galois closure of Abhyankar's equation $y^{p+s} - x y^s + 1$ is an $A_{p+s}$ cover of the affine line with 
inertia groups of order $p(p-1)s/{\rm gcd}(p-1, s(s+1))$.  The factor of ${\rm gcd}(p-1, s(s+1))$ is lost from 
applying Abhyankar's Lemma to remove tame ramification.  Even when this factor is trivial, this type of equation 
only produces inertia groups $I \simeq \langle \tau \rangle \rtimes \langle \beta \rangle$ where $\beta$ acts on $\{p+1, \ldots, p+s\}$ 
as an $s$-cycle.

Here is a very special case of Conjecture~\ref{Iconjecture}.
\begin{question}
If $p \geq 5$, is Abhyankar's Inertia Conjecture true for the group $A_{p+1}$? 
\end{question}

We highlight this as a special case because Abhyankar's treatment of the group $A_{p+1}$ 
differed from that of other alternating groups.  
To realize $A_{p+1}$ as a Galois group of a cover of the affine line, Abhyankar constructed an $A_{p+2}$-Galois cover $Y \to \PP^1$
which factors through a degree $p+2$ cover $X=\PP^1 \to \PP^1$.  The $A_{p+1}$ Galois 
subcover $Y \to X$ is ramified at two points, but the 
tame ramification above one point can be removed using Abhyankar's Lemma.  Unfortunately, the application of Abhyankar's lemma
also eliminates the tame ramification at the wildly ramified point.  The inertia groups of the resulting $A_{p+1}$ cover of the affine line are of order $p$.

\section{Arithmetic covers} \label{SConj3}

In the previous sections, we considered curves over 
an algebraically closed field $k$ of characteristic $p > 0$. 
There, we explored how the triviality of the prime-to-$p$ fundamental group places restrictions on the finite groups $G$ which can occur as the Galois group of a cover $\phi:Y \to \PP^1$ branched only at $\infty$ as well as on the inertia groups that can occur over $\infty$.
These restrictions over algebraically closed fields $k$ also limit the finite groups $G$ that can occur as covers of curves defined over non-algebraically closed fields $\ell$.  
As usual, Abhyankar conjectured that the covers that do exist are exactly those that can exist under these limitations.
When working with covers of curves over non-algebraically closed fields, it is convenient to distinguish between the part of the cover coming from an extension of the scalars and the part that would remain the same if the scalars were extended.  
This leads naturally to the definitions of the arithmetic and geometric fundamental group and to Abhyankar's Arithmetical Affine Conjectures.

\subsection{Background on covers over perfect fields.}\label{Sperfectfields}

Let $\ell$ be a perfect field. 
Let $X$ be a smooth projective $\ell$-curve.  
We call a cover of $\ell$-curves $Y \to X$ {\it regular} if the algebraic closure of $\ell$ in the function field $K_Y$ of $Y$ is $\ell$ itself (some authors use ``geometric").  
For algebraically closed fields this is of course always the case, however in general the extension may have some ``arithmetic" component.

Let $\ell$ be a field with $\ell'/\ell$ a finite Galois extension.  
Let $Y$ be a curve defined over $\ell'$ and let $A$ be a subgroup of $\text{Aut}_\ell(Y)$ such that $A$ induces the full group of $\ell$-automorphisms of $\ell'$.  
Then $A$ induces a Galois cover $Y \to X:=Y/A$ with Galois group $A$ where $X$ is a curve defined over $\ell$.
Let $G$ be the normal subgroup of $A$ which is trivial on $\ell'$.  
So $A/G \cong {\rm Gal}(\ell'/\ell)$.  
Then there is a Galois cover $\phi:Y \rightarrow X_{\ell'}:= X \times_\ell \ell'$ that is regular (over $\ell'$) and has Galois group $G$.  
Let $\Delta$ be the set of branch points of the cover $Y \to X$ and note that this set is $A$-invariant and so defined over $\ell$.  
Let $\Delta_G$ be the set of branch points of $Y \to X_{\ell'}$.  This set is defined over $\ell'$.

Given any Galois cover $\phi:Y \to X$ with Galois group $A$ and branch locus ${\Delta}$ defined over $\ell$, let $\ell'$ be the algebraic closure of $\ell$ in $K_Y$.
Then $\phi:Y \to X$ factors through $X_{\ell'} \to X$,  $Y \to X_{\ell'}$ is a Galois cover defined over $\ell'$, and we have a short exact sequence
$$1 \to {\rm Gal}(Y/X_{\ell'}) \to A \to {\rm Gal}(\ell'/\ell) \to 1.$$
Notice that the algebraic closure of $\ell'$ in $K_Y$ is $\ell'$ again. 
Thus $Y \to X_{\ell'}$ is regular and we call $G={\rm Gal}(Y/X_{\ell'})$ the {\it geometric Galois group} of $Y \to X$. 

If we begin with a regular non-Galois cover $Y \to X$ then its Galois closure $\bar{Y} \to X$ may or may not be regular. 
Let $A$ and $G$,  respectively be the arithmetic and geometric Galois groups of the cover $\bar{Y}\to X$, then $A$ is called the {\it arithmetic monodromy  group} of $Y \to X$ and $G$ the {\it geometric monodromy group} of $\bar{Y} \to X$.

This breakdown of the arithmetic and geometric monodromy groups can be extended to the level of fundamental groups. 
Let $U_{\bar \ell} = U \times_\ell {\bar \ell}$.  
Then we have a short exact sequence analogous to the one that exists on a finite level.  Namely,
$$1 \to \pi_1(U_{\bar \ell}) \to \pi_1(U) \to {\rm Gal}({\bar \ell}/\ell) \to 1$$
and it splits if there exists a $\ell$-point on $U$.
The group $\pi_1(U_{\bar \ell})$ is called the {geometric fundamental group}.  
In the earlier parts of this survey article, the curves in question have been defined over an algebraically closed field and thus the fundamental group has
been the geometric fundamental group.

\subsection{Abhyankar's Affine Arithmetical Conjecture}\label{SArithmeticalConj}

The 1990's brought much progress on understanding the set of finite quotients of the arithmetic and geometric fundamental group.  
For curves over algebraically closed fields, where the two groups coincide, Raynaud and Harbater's proof of Abhyankar's Conjecture determined the set completely.
Nevertheless,  because their proof was ``existential," Abhyankar continued to ``to march on with the project of finding explicit nice equations for nice groups...."  \cite{Abh01}.
He set about computing the geometric monodromy groups (Galois groups) of the two families of equations mentioned in his 1955 and 1957 papers.  
For a field $l$ of characteristic $p>0$ (not necessarily algebraically closed), these families were defined by the following equations, 
$${\overline f} = h(X,Y^p)Y^t-aX^rY^t + bX^s$$
and 
$${\tilde f} =  h(X,Y^p)- aX^rY^t + bX^s$$
for non-zero constants $a$ and $b$, and 
where $h(X,Y) \in \ell[X,Y]$ is a monic polynomial of degree $\nu$ in $Y$.
By considering the $Y$-discriminant of these polynomials (viewed as polynomials in $Y$ with coefficients in $K(X)$), one finds the following:
\begin{enumerate}
\item If $r \geq 0$, $s=0$, and $t \not\equiv 0 \pmod{p}$ then the
  equation $\overline f =0$ gives an unramified covering of the affine
  $\ell$-line (i.e., the projective line minus one point).
\item  If $r \geq 0$, $s>0$, and $t \not\equiv 0 \pmod{p}$ then the equation $\overline f =0$ gives an unramified covering of the punctured affine $\ell$-line.
\item  If $h(X,0)=0$, $r \geq 0$, $s\geq 0$, and $t \not\equiv 0 \pmod{p}$ then the equation $\tilde f =0$ gives an unramified covering of the punctured affine $\ell$-line.
\end{enumerate}
In \cite{A22}, Abhyankar showed that for various choices of the parameters $a$, $b$, and $t$ and for suitable primes $p$ and prime powers $q$, the Galois groups realized using these covers included the symmetric and alternating groups $A_n$ and $S_n$ of any degree $n>1$, the linear groups ${\rm GL}(m,q)$, ${\rm SL}(m,q)$, ${\rm PGL}(m,q)$, ${\rm PSL}(mq)$, ${\rm AGL}(1,q)$ and $\F_q^{+}$ for any integer $m>1$, the Mathieu groups $M_{11}$, $M_{12}$, $M_{22}$, $M_{23}$, and $M_{24}$, and the automorphism group of $M_{12}$.  
The geometric monodromy (Galois) groups of covers of the affine line and the punctured affine line found in \cite{A22}, combined with Harbater and Raynaud's proof of Abhyankar's conjecture, lead Abhyankar to state the following conjecture:
\begin{conj}[Early Affine Arithmetical Conjecture] \cite[Conjecture 9.2C]{A22}\label{eaconjecture}
Let $k$ be an algebraically closed field of characteristic $p>0$. For a finite group $A$,  if $\pi_1(\A^1_{k}) \twoheadrightarrow A$
then 
$\pi_1(\A^1_{\F_p}) \twoheadrightarrow A$.  
\end{conj}
\noindent Abhyankar went on to ask whether the finite quotients of $\pi_1(\A^1_{\F_p})$ 
might not be exactly the finite quotients of $\pi_1(\A^1_{k} - \{0\})$.

To realize the groups above as Galois groups, Abhyankar used results from several of his 1980's era papers as well as the first paper in a series under the theme ``nice equations for nice groups" (\cite{A21}, \cite{A23}-\cite{A26}).
In that series, he gave natural representations of classical groups as Galois groups of factorizations of generalized Artin-Schreier polynomials defined over an algebraically closed field $k$ of characteristic $p>0$.
In \cite{A27} Abhyankar provided a unified view, or ``mantra" as he preferred to call it, of the groups arising from ``nice" equations.  
These were again Galois groups of covers of the affine line and the punctured affine line defined over an algebraically closed field of characteristic $p>0$, and they included some symplectic, orthogonal and unitary groups.
In the late 1990's, motivated by M.~Fried,  Abhyankar went back to his ``nice" equations, dropping the algebraically closed hypothesis, and computed the {\it arithmetic} monodromy groups of the covers defined in those papers (\cite{A33}, \cite{A34}). 
He noticed that in the cases where the geometric monodromy is a linear or symplectic group, the arithmetic monodromy group is the corresponding semi-linear group. 
For example, in cases where the geometric monodromy group of the cover over $\A^1_{\F_p}$ defined by a nice equation is ${\rm PGL}(m,q)$, the arithmetic monodromy group is the quotient by scalar matrices of the semi-direct product of ${\rm Aut}(\F_q)$ and ${\rm GL}(m,q)$.

The results in \cite[Prop. 4.2.4]{A33} were sufficient for Abhyankar to upgrade his question to a conjecture:
\begin{conj}[Affine Arithmetical Conjecture] \cite[Section 16]{Abh01}\label{aconjecture}
Let $k$ be an algebraically closed field of characteristic $p>0$. For a finite group $A$,  
$\pi_1(\A^1_{\F_p} ) \twoheadrightarrow A$ 
if and only if
$\pi_1(\A^1_{k}- \{0\}) \twoheadrightarrow A$.
\end{conj}

\noindent Notice that if $A$ satisfies the first condition, then it is cyclic-by-quasi-$p$
by the fundamental exact sequence (given at the end of the previous subsection); and therefore satisfies the second condition by Abhyankar's Conjecture for the punctured affine line.  The part of the conjecture that remains to be shown is therefore the other direction.  In work with Paul Loomis, Abhyankar showed that in many cases the results in the earlier papers still held when the condition that $k$ is algebraically closed is replaced by the condition that the field contain $\F_q$ for an appropriate $p$-power $q$ (\cite{AL1}, \cite{AL2}).
This led Abhyankar in 2001 to extend his conjecture to a function field version, following the ``as large as possible" mantra which is the common thread of all his conjectures.  

To state this version we introduce a little notation (cf.\ Section \ref{SConj1}).  
Let $\pi_A(\A^1_{\F_p,\infty})$ be the set of all Galois groups of finite Galois extensions of the rational function field $\F_p(X)$ (i.e. the finite quotients of the absolute Galois group of $\F_p(X)$).  This notation was meant to convey that $\Spec(\F_p(X))$ is the result of deleting all the closed points from $\A^1_{\F_p}$.

\begin{conj}[Total Arithmetical Conjecture] \cite[Section 16]{Abh01}\label{atconjecture}
Every finite group lies in $\pi_A(\A^1_{\F_p,\infty})$.
\end{conj}

This conjecture is equivalent to asking for a positive answer to the inverse Galois problem over $\F_p(X)$.  Like Conjecture~\ref{aconjecture}, it remains open. In fact, for every prime power $q$, the inverse Galois problem over $\F_q(X)$ remains open as well (and indeed it remains open over every global field).  On the other hand, there is a partial result in this direction, due to 
Fried-V\"olklein, Jarden, and Pop: for each finite group $G$ there is an integer $n$ such that for all prime powers $q > n$ the group $G$ is a Galois group over $\F_q(X)$. See \cite{Har03}, Prop.~3.3.9 and Remark~3.3.11, for more about that result.  The Total Arithmetical Conjecture is also related to the Higher Dimensional Conjecture which is discussed in the next section.

\subsection{A related result}

Let $A$ be a finite group generated by $p(A)$ and one other element $a \not\in p(A)$.  This is a cyclic-by-quasi-$p$ group, so it occurs as a Galois group of a cover of the punctured affine line by Abhyankar's Conjecture for affine curves.  Following Abhyankar's Affine Arithmetical Conjecture, it should also occur as the arithmetic monodromy group over $\A^1_{\F_p}$.  The next result does not prove the Affine Arithmetical Conjecture for this type of group, but it does show that there exists an finite extension $\ell/\F_p$ over which $A$ does occur.

\begin{theorem}\cite[Theorem 4.5]{GS} {\label{GS-aconjecture}}
  Let $A$ be a finite group generated by $p(A)$ and one other element
  $a \not\in p(A)$.  Then there exists a finite extension $\ell/{{\F}_p}$ and a Galois cover $\psi :Z \to {\P}^1_{\ell}$ such
  that the arithmetic Galois group of $\psi$ is $A$, the geometric Galois group
  is $p(A)$, and the branch locus consists of only one $\ell$-rational
  point.  Moreover, the resulting regular Galois cover with Galois
  group $p(A)$ is a cover of the projective $\ell'$-line for some finite
  extension $\ell'/\ell$.
\end{theorem} 

The proof of this theorem uses Harbater's proof of Abhyankar's Conjecture, a descent argument, and elementary group theory.  It provides a nice connection between two of Abhyankar's conjectures, and we give a sketch of it here that summarizes the results found in Section 4 of \cite{GS}.  Before sketching the proof a little notation is needed, in addition to the notation of Section \ref{Sperfectfields}, which we preserve.

Let ${\Delta}_{\bar  \ell}={\Delta} \times_\ell \bar{\ell} = \{\tau_1,...,\tau_r\} $, 
$Y_{\bar{\ell}}= Y \times_\ell \bar{\ell}$ and $X_{\bar{\ell}}= X \times_\ell\bar{\ell}$.  
We say that $Y \to X$ has {\it inertial description} $(H_1,...,H_r)$ if $Y_{\bar{\ell}} \to X_{\bar{\ell}}$ has the property that over every $\tau_i$ there exists a $\hat{\tau_i}$ with inertia group $H_i$.  
This is defined only up to individual conjugation.  
Notice that if $\tau \in {\Delta}$ is a closed point in $X$ and $\tau_i$ and $\tau_j$ in $X_{\bar{\ell}}$ are conjugate geometric points over $\tau$ then the inertia groups of $\hat{\tau_i}$ and $\hat{\tau_j}$ over $\tau$ are $H_i$ and $H_j$ respectively and they are conjugate to one another in $A$, but not necessarily in $G$.

\begin{proof} Let $S$ be a $p$-Sylow subgroup of $A$.  
By the result of  Schur-Zassenhaus, we may assume that $A $ is the semi-direct product $p(A)  \rtimes \langle a \rangle$ with $\langle a \rangle \subset N_A(S)$.  
Let $X_0 = {\mathbb  P}_{\bar{\F}_p}^1$.  
By Abhyankar's Conjecture (as in Proposition 4.1 of \cite{Har03}), we know that $A$ is  the Galois group of a Galois cover $Y_0 \to X_0$ such that  $\Delta_{Y_0} = \{\delta_0, \sigma_0 \}$ is the branch locus.
Moreover, by that result, we may assume that the inertial  description of this cover is $(S \rtimes \langle a \rangle, \langle  a^{-1} \rangle)$.  
Notice that the subcover $Y_{0,p(A)} \to X_0$ induced by $p(A) \triangleleft A$ is $p'$-cyclic, branched at only two points, and is totally ramified over these two points.  
Thus $Y_{0,p(A)} \cong {\P}_{\bar{\F}_p}^1$.

The cover $Y_0 \to X_0$ descends to a cover $Y_1 \to X_1$ defined over a finite extension $\ell$ of ${\F}_p$ with branch
locus $\{\delta_1, \sigma_1 \}$ consisting of $\ell$-rational points  and the same inertial description.  
Moreover, $X_1 = {\P}_{\ell}^1$, and there is 
finite field extension ${\ell}'/{\ell}$ which is 
Galois with Galois group $B:=A/p(A)$.  
The subcover $Y_{1,p(A)} \to X_1$ induced by $p(A)  \triangleleft A$ is $p'$-cyclic, branched at only two points, and is  totally ramified over these two points (since the inertia groups do not change).  
These points will also be ${\ell}$-rational.  
Thus  $Y_{1,p(A)}$ has genus zero and a ${\ell}$-rational point, so  $Y_{1,p(A)}$ is ${\ell}$-rational.
  
Let $M_1$ and $K_1$ be the function fields of $Y_{1} $ and $X_1$, respectively. 
We now consider the field $M = M_1{\ell}'$. 
Since ${\ell}'K_1$ and $M_1$ are linearly disjoint over $K_1$, the extension $M/K_1$ is Galois with 
${\rm Gal}(M/K_1) = A\times B$, and it corresponds to an $A \times B$-cover $Z \to X_1$ with inertial description $(S \rtimes \langle a \rangle, \langle  a^{-1} \rangle)$.

Let $\pi : A \to A/G = B$ be the natural projection where $G=p(A)$.
Consider the injective group homomorphism $\iota:A \to A \times B$ by $a \mapsto (a,\pi(a))$.
Let $A_0$ denote the image of $\iota$.
Thus $A \cong A_0$. 
Let $K$ be the fixed field $M^{A_0}$ of $A_0$ acting on $M$.

The field $K$ is a finite extension of the rational function field ${\ell}(x)$.
Thus, $K$ is the function field of some curve $W$ over ${\ell}$ and $M/K$ corresponds to a cover $\psi:Z \to W$.
Note that ${\ell}'$ is the algebraic closure of ${\ell}$ in $M$.
Let $G_0$ denote the image of $G$ in $A_0$ under $\iota$. 
Since $\pi(G)$ is trivial, $G_0$ is contained in the subgroup of $A_0$ fixing $K{\ell}'$. 
Now  $[K{\ell}': K] = [{\ell}': {\ell}] = [A: G] = [A_0: G_0]$, 
so $G_0$ is the subgroup of $A_0$ fixing ${\ell}'$.
Thus $G_0 \cong G$ is the geometric Galois group of the cover $\psi:Z \to W$.

Now we have constructed a cover $Z \overset{\psi}{\rightarrow} W \to X_1$ in which the cover $\psi$ has arithmetic Galois group $A$ and geometric Galois group $p(A)$.  
Since $\psi$ factors through $Z/(G \times 1) = (Y_{1,p(A)})_{\ell'}$, which is 
$\ell'$-rational, it follows that $W$ is $\ell$-rational.
By \cite[Proposition~4.1(3)]{GS}, for any branch point $\beta \in W$, its inertia group (with respect to $\psi$) is $A$-conjugate to $I \cap p(A)$ where $I$ is an inertia group of $Y_1 \to {\P}_{\ell}^1$.
If $\beta$ lies over $\delta_1 \in X_1$ then its inertia group is  $A$-conjugate to $(S \rtimes \langle a \rangle) \cap p(A) = S$.  
If $\beta$ lies over $\sigma_1$ then its inertia group is $A$-conjugate to $ \langle a \rangle \cap p(A) = e$.
  
Combining all of the above, we have a Galois cover $\psi :Z \to W={\mathbb P}_{\ell}^1 $ with arithmetic Galois group $A$ and geometric Galois group is $G = p(A)$.  
The branch locus consists of only one point with inertia group $S$.  The resulting $p(A)$-Galois cover is a cover of the projective ${\ell}'$-line.
\end{proof}

One could also use the proof of Theorem B from \cite{P95} to prove the above result.  A sketch of that proof can be found in \cite[Remark 4.6]{GS}.

\section{Higher dimensional conjectures} \label{dShighdim}

Following the proof of his conjecture for affine curves over algebraically closed fields of characteristic $p$, Abhyankar considered not only analogous conjectures for affine curves over finite fields (see Section~\ref{SConj3} above) but also analogs for higher dimensional affine varieties over algebraically closed fields of characteristic $p$. 
In doing so, he was guided by the informal form of his conjecture (Conjecture~\ref{dInformalAC} above).  As in the curve case, this conjecture can be viewed as having two parts:
\begin{enumerate}
\item For any finite group $G$, one has $G\in\pi_A(X)$ if and only if $G/p(G)\in\pi_A(X)$. \label{d-qpextension}
\item If $X_0$ is an analogous complex variety, then a prime-to-$p$ group $G$ is in $\pi_A(X)$ if and only if it is in $\pi_A(X_0)$. \label{d-analogous}
\end{enumerate}

Part~(\ref{d-qpextension}) in particular asserts that every quasi-$p$ group is in $\pi_A(X)$.  Of course, this is not the case for all $X$ in characteristic $p$; e.g.\ if $k$ is algebraically closed then 
it fails for $\Spec\bigl(k[[x,y]]\bigr)$ and $\Spec\bigl(k((x))\bigr)$, as well as for smooth projective $k$-curves. Namely, only the trivial group occurs over $k[[x,y]]$, by Hensel's Lemma; only certain solvable groups occur over $k((x))$, since the full Galois group is the inertia group (see Lemma~\ref{1SRH}(\ref{inertia structure part}) and the comment after Definition~\ref{decomp inertia def}); and only certain groups on $2g$ generators occur over a smooth projective curve of genus $g$, as discussed in Section~\ref{Sfundgroups}.  To avoid these cases, Abhyankar said that his generalized conjecture would apply only for spaces in which there is ``enough room'' for all quasi-$p$ groups to arise as Galois groups.  

As for curves, the meaning of part~(\ref{d-analogous}) depends on interpreting what is meant by ``an analogous complex variety''.  Abhyankar considered in particular the following situations, where $k$ is algebraically closed of characteristic $p$: 

\begin{enumerate}
\renewcommand{\theenumi}{\roman{enumi}}
\renewcommand{\labelenumi}{(\roman{enumi})}\item \label{d-highdimlocal}
$X = \Spec\bigl(k[[x_1,\dots,x_n]][x_1^{-1},\dots,x_t^{-1}]\bigr)$, where $n \ge 2$ and $t \ge 1$.
\item \label{d-highdimglobal}
$X$ is the complement of a set of $t+1$ hyperplanes in general position in $\P^n_k$, where $n \ge 2$ and $t \ge 1$  Here, ``general position'' means that the union of these hyperplanes is a normal crossing divisor.
\end{enumerate}

In the situation of (\ref{d-highdimlocal}) the complex analog is 
$\Spec\bigl(\C[[x_1,\dots,x_n]][x_1^{-1},\dots,x_t^{-1}]\bigr)$; while in the situation of (\ref{d-highdimglobal}), the complex analog is the complement of $t+1$ hyperplanes in general position in $\P^n_\C$.  In each of these cases, the groups that occur over the complex analog are precisely the abelian groups that have a generating set of at most $t$ elements.  (See \cite{Zar29}, Section~7, for the global case.)
So in these two cases, Abhyankar's generalized conjecture says
that $G \in \pi_A(X)$ if and only if $G/p(G)$ is an abelian group having a set of at most $t$ generators.  See \cite{Abh97}, Conjectures~2.2 and~3.2.  As Abhyankar said there, in these two cases it had been shown in \cite{Abh55} (resp.\ \cite{Abh59} and sequels) that every group in $\pi_A(X)$ is of this form, and that he had implicitly posed this conjecture in those earlier papers.  Abhyankar discussed these conjectures further in \cite{Abh01} and \cite{Abh02}.

While the above conjectures are stated in the context that $k$ is algebraically closed of characteristic $p$, assertion~(\ref{d-qpextension}) makes sense even if $k$ is finite.  Indeed, Abhyankar's Affine Arithmetical Conjecture~\ref{aconjecture} is simply the statement of this assertion in the case that $X$ is the affine line over $\F_p$.  Similarly, assuming that every prime-to-$p$ group is a Galois group over $\F_p(X)$ (which is known for $p=2$, by the theorems of Shafarevich and Feit-Thompson),  
Abhyankar's Total Arithmetical Conjecture~\ref{atconjecture} is just the statement of~(\ref{d-qpextension}) for $\Spec\bigl(\F_p(X)\bigr)$.  While those conjectures are open, they are believed to hold, and this may be taken to be reason to expect that the above conjectures would hold in the cases (\ref{d-highdimlocal}) and (\ref{d-highdimglobal}).

But in fact, in (\ref{d-highdimlocal}) and (\ref{d-highdimglobal}), it turns out that the classes $\pi_A(X)$ are in general strictly smaller than conjectured, because there are additional conditions that must be satisfied in order for a group to be a Galois group in those settings.  Moreover the same holds in the arithmetic situation, for the punctured affine line.  Specifically, there is the following assertion.  (As usual, a {\em supplement} to a normal subgroup $N$ in a group $G$ is a subgroup $H$ of $G$ such that $G$ is generated by $N$ and $H$.)

\begin{prop} \label{d-highdimgpconditions}
Let $X$ be one of the following:
\begin{enumerate}
\item \label{d-hdproplocal}
$\Spec\bigl(k[[x_1,\dots,x_n]][x_1^{-1},\dots,x_t^{-1}]\bigr)$, with $1 \le t \le n$, where $k$ is algebraically closed of characteristic $p$.
\item \label{d-hdpropglobal}
The complement of a set of $t+1$ hyperplanes in general position in $\P^n_k$,
where $0 \le t \le n$.
\item 
$\A^1_k - \{0\}$, where $k$ is a finite field of characteristic $p$.
\end{enumerate}
If $G \in \pi_A(X)$, then $p(G)$ has a prime-to-$p$ supplement that lies in $\pi_A(X)$.  Moreover this supplement may be chosen so as to normalize a non-trivial $p$-subgroup of $G$ provided that $p$ divides the order of $G$; and the set of $p$-subgroups that can be normalized by such supplements together generate $p(G)$.
\end{prop}

The result in these three cases was shown in~\cite{HarvdP02}, at 
Theorem~3.3, Theorem~4.5 (and its proof), and Theorem~5.5.
Actually, the same result holds as well in the case that $X$ is an affine curve over an 
algebraically closed of characteristic $p$ --- i.e.\ in the case of Abhyankar's original conjecture.  But there the extra conditions are automatic for $G$, once we know that $G/p(G)$ is in $\pi_A(X)$, because the class of prime-to-$p$ groups in $\pi_A(X)$ consists of {\em all} prime-to-$p$ groups on a specified number of generators.  See \cite[Example~5.1]{HarvdP02}.  And in fact, these conditions were used in the proof of Abhyankar's conjecture for affine curves (which makes plausible that they are needed in general for a group to lie in $\pi_A(X)$).

On the other hand, in the three cases listed in the proposition above, these conditions are not automatic, provided $t>1$ in the first two cases.  Indeed, there are groups $G$ such that $G/p(G)$ lies in $\pi_A(X)$ but where $G$ does not satisfy those extra conditions.  See Examples~5.2, 5.3, and 5.4 of \cite{HarvdP02}.  Thus the higher dimensional conjectures proposed by Abhyankar in \cite{Abh97}, as well as the ``obvious'' conjecture for the punctured affine line, do not hold.  This does not really contradict Abhyankar's ``maximal'' philosophy, since this simply shows that additional groups can be ruled out.  But the precise condition that must be assumed on $G$ remains open.  See~\cite{HarvdP02} and \cite{HarUnpub02}. 
 
The above Proposition suggests that conditions~(\ref{d-qpextension}) and (\ref{d-analogous}) above may determine $\pi_A(X)$ in those cases where the extra conditions in Proposition~\ref{d-highdimgpconditions} are automatic, as in the case of affine curves over an algebraically closed field.  Whether those conditions are automatic is a group-theoretic problem concerning the class $\calC$ of prime-to-$p$ groups in $\pi_A(X)$.  These conditions are automatic if a class $\calC$ is closed under Frattini extensions \cite[Theorem~6.2]{HarvdP02}.  On the other hand, the conditions are
not automatic if $\calC$ is closed under taking normal subgroups (and quotients) but not under taking Frattini extensions \cite[Theorem~6.3]{HarvdP02}.
They are also not automatic if $\calC$ is a class of abelian groups that contains a rank two elementary abelian group, or if $\calC$ is the class of metacyclic prime-to-$p$ groups \cite[Theorem~6.1]{HarvdP02}.  

In cases (\ref{d-hdproplocal}) and (\ref{d-hdpropglobal}) of Proposition~\ref{d-highdimgpconditions}, the conditions in the conclusion of the proposition impose new restrictions for a group $G$ to be in $\pi_A(X)$ (beyond $G/p(G)$ being in $\pi_A(X)$) only if $t>1$.  If $t=0$ (resp.\ $t=1$) then the prime-to-$p$ groups in $\pi_A(X)$ are just the trivial group (resp.\ the cyclic prime-to-$p$ groups), and so those conditions are automatic.  The above discussion then suggests that Abhyankar's conjecture should hold in (\ref{d-hdproplocal}) if $t=1$ and in (\ref{d-hdpropglobal}) if $t=0,1$ (the case $t=0$ being excluded from (\ref{d-hdproplocal}) because of not having ``enough room'').  That is, a group $G$ is a Galois group over $X$ if and only if $G/p(G)$ is.  
And indeed, this is the case for (\ref{d-hdpropglobal}) if $t=0,1$, using
Abhyankar's conjecture for curves together with the isomorphisms
$\A^n \cong \A^1 \times \A^{n-1}$ and $\A^n - H \cong (\A^1 - \{0\}) \times \A^{n-1}$
(where $H$ is a hyperplane).  The case for (\ref{d-hdproplocal}) with $t=1$ is less obvious, but again it turns out that Abhyankar's conjecture holds.  That is, a finite group $G$ is a Galois group over $\Spec\bigl(k[[x_1,\dots,x_n]][x_1^{-1}]\bigr)$ for $n>1$ if and only if $G/p(G)$ is
cyclic.  This was shown in \cite[Theorem~3.3]{HaSt03}, by deducing the result from a more global assertion.  These results provide evidence that Abhyankar's Conjecture holds for $X$ if and only if the conditions in Proposition~\ref{d-highdimgpconditions} are automatic for the class of prime-to-$p$ groups in $\pi_A(X)$.

Finally, we return to the arithmetic conjectures posed in Section~\ref{SConj3}.  In the case of the affine line over $\F_p$, the prime-to-$p$ groups in $\pi_A(\A^1_{\F_p})$ are precisely those that are cyclic (with generator given by Frobenius), and so the conditions in Proposition~\ref{d-highdimgpconditions} are automatic.  This suggests that the Arithmetical Affine Conjecture~\ref{aconjecture} is correct.  In the function field situation, 
if every prime-to-$p$ group is a Galois group over $\F_p(X)$ (as expected, and as known for $p=2$), then the conditions in Proposition~\ref{d-highdimgpconditions} would be automatic for $\Spec\bigl(\F_p(X)\bigr)$, and this would provide additional evidence for the Abhyankar's Total Arithmetical Conjecture~\ref{atconjecture}.

\section {Projective covers} \label{Sproj}

This section concerns fundamental groups of projective curves over algebraically closed fields. As with affine curves, the fundamental group is known in characteristic $0$.  In the case of a projective curve of genus $g$ over an algebraically closed field of characteristic $0$, the fundamental group is the profinite group generated by elements $a_1, b_1, \dots, a_g, b_g$ subject to the sole relation $\prod_{i=1}^g [a_i,b_i] = 1$.  Notice that this is a finitely generated profinite group, but unlike the affine case, it is not free. 

In characteristic $p>0$, the distinction between affine and projective cases is more pronounced.  For example, the fundamental group (rather than just its finite quotients) is {\it known} in two cases: $\pi_1(\P^1_k)$ is trivial, and the fundamental group of an elliptic curve depends only on the Hasse-Witt invariant.  But much less is known about the fundamental group or its finite quotients for projective curves of genus $\geq 2$. 

As noted earlier, Abhyankar conjecture for affine curves does not extend to the case of projective curves (where $r=0$). For example, given an elliptic curve $E$ in characteristic $p$, the group $(\Z / {p\Z})^2$ is not a finite quotient of $\pi_1(E)$ despite the fact that its maximal prime-to-$p$ quotient (namely, the trivial group) is a quotient of $\pi_1(E)$. Nevertheless, by \cite[XIII, Cor. 2.12]{SGA1} the forward implication of Abhyankar's conjecture does hold for projective $k$-curves;
in fact $\pi_A$ for a projective curve in characteristic $p$ is contained in $\pi_A$ for a curve of the same genus in characteristic zero (see Remark~\ref{dRproj}). 

To discuss this in more detail, we go back to affine curves, and restrict our attention to covers for which the projective completion of the cover is tamely ramified.

Recall the following from Section \ref{Sdecomptame}. 
Let $D$ be a projective $k$-curve of genus $g$.  
Let $\lambda$ be a point on $D$. 
A cover $X \to D$ is called {\it tamely ramified} at a point $\tau$ of $X$ lying over $\lambda$ if the ramification index at $\tau$ is prime to $p$.
We say that $X \to D$ is {\it tame over} $\lambda$ if it is tame at every point lying over $\lambda$, and we say that $X \to D$ is {\it tame} if for every point in $D$ all points lying over it are either unramified or tamely ramified.
Obviously, if $p$ does not divide the order of a group $G$, then any $G$-Galois cover is tame.  
However there are tame covers with Galois group not of order prime to $p$. For example, unramified $\Z/p\Z$-Galois covers of an ordinary elliptic curve are tame.  

Consider the set of all finite groups occurring as Galois groups of Galois covers of $D-\{\lambda_1 , \lambda_2 , ..., \lambda_r\}$ whose projective completion is tame.  
This set can be given the structure of an inverse system.  
The inverse limit is a quotient of $\pi_1(D-\{\lambda_1 , \lambda_2 , ..., \lambda_r\})$ denoted by 
$\pi_1^t(D-\{\lambda_1 , \lambda_2 , ..., \lambda_r\})$.
Comparing tame covers to analogous covers in characteristic zero brings us to our
first necessary condition for a finite group to occur as a Galois group over $D$.

\subsection{Condition I}
Given a projective curve $Y$ of genus $g$ defined over an algebraically closed field of characteristic $0$, let $\{\gamma_1,..., \gamma_r\}$ be closed points on $Y$.  The fundamental group depends only on $g$ and $r$, and we will denote it by $\hat{F}_{g,r}$.
If we fix a prime $p$, we can define $p$-tame covers of $Y$ as above and also form the $p$-tame quotient of the fundamental group.  
This will depend only on $g$, $r$, and $p$, and we denote it by $\hat{F}_{g,r}^{p-tame}$.
Any tame cover of a projective $k$-curve $D$ of genus $g$ can be lifted to a $p$-tame cover of a genus $g$ curve defined over an algebraically closed field of characteristic zero 
\cite{SGA1}.  
As a consequence, one obtains the following result, discussed in Section~\ref{Sfundgroups} above. 

\begin{proposition}\cite[XIII, Corollary 2.12]{SGA1} \label{SGA1Cor212}
Given a (possibly empty) set $\{\lambda_1 , \lambda_2 , ..., \lambda_r\}$ of
closed points of a projective $k$-curve $D$ of genus $g$, the group  $\pi_1^t(D-\{\lambda_1 , \lambda_2 , ..., \lambda_r\})$ is a quotient of $\hat{F}_{g,r}^{p-tame}$.
\end{proposition}

Since an unramified cover of $D$ is tame, Proposition \ref{SGA1Cor212} implies that if $G$ is a finite quotient of $\pi_1(D)$ it is necessary that $G$ is a quotient of $\hat{F}_{g,0}^{p-tame}=\hat{F}_{g,0}$.
Let $\pi^t_A(D-\{\tau_1 , \tau_2 , ..., \tau_r\})$ denote the set of isomorphism classes of (continuous) finite quotients of $\pi^t_1(D-\{\tau_1 , \tau_2 , ..., \tau_r\})$. Then $\pi^t_A(D-\{\tau_1 , \tau_2 , ..., \tau_r\})$ contains $\pi_A(D)$. However, this containment is not as useful as it might appear because $\pi^t_A$ is known only for $(g, r)=(0, 0), (0, 1), (0, 2)$ and $(1, 0)$. (Notice that this is despite the fact that we know $\pi_A$ for any genus as long as $r>0$.) 

Nevertheless, Proposition \ref{SGA1Cor212} does give important information on the group structure of $\pi_1^t$ and hence for $\pi_1(D)$.   
In particular, the group $\pi_1^t$ is a finitely generated profinite group, and as such it is determined by its finite quotients \cite[Proposition 15.4]{FJ}. 
This is interesting because in general the same set of groups can form different inverse systems, and hence different inverse limits. 
Indeed, as was mentioned in Section \ref{SAbhHist}, 
$\pi_1$ of an affine curve in characteristic $p$ is not determined by 
$\pi_A$ (which is not finitely generated as a topological group).
For example, while $\pi_A(E-\tau) = \pi_A(\P^1_k - \{0, 1, \infty\})$ where $E$ is an elliptic curve, it is known that one can find points $\lambda$ and $\lambda'$ on the projective line such that
$\pi_1^t(\P^1_k - \{0, 1, \infty, \lambda\}) \neq \pi_1^t(\P^1_k - \{0, 1, \infty, \lambda'\})$ 
(\cite{Bouw}, \cite{Tam03}).
Moreover, according to Grothendieck's Anabelian Conjecture \cite{Gr84} the full $\pi_1$ should determine the curve; and there is no viable conjecture for the set of finite quotients of $\pi_1^t$.
Returning to the projective case, since $\pi_1^t(D) = \pi_1(D)$, we have that $\pi_A(D)$ determines $\pi_1(D)$, and again there is no viable conjecture for the set of finite quotients in terms of discrete invariants.

\subsection{Condition II}
Given a group $G$ lying in $\pi_A(D)$, we consider next a condition on $G$ which arises from the Hasse-Witt invariant. 
For a finite group $G$, let $\sigma(G)$ denote the $p$-rank of the abelianization of $G$, and let $A$
be the maximal elementary abelian $p$-quotient of $G$. 
Then $A$ has rank  $\sigma(G)$.
Suppose that there exists a $G$-Galois cover $X \to D$. 
Then $X \to D$ factors through an $A$-Galois cover of $D$. 
Let $\gamma(D)$ be the Hasse-Witt invariant of $D$. 
This is an integer between $0$ and the genus $g$ of $D$ (equal to the $p$-rank of the Jacobian $J(D)$ of $D$), and the $p$-part of the abelianization of $\pi_1(D)$ is isomorphic to $\Z_p^{\gamma(D)}$. 
Thus if $G \in \pi_A(D)$, it is necessary that $\sigma(G) \leq g$. 

Given $p$-group $G$ of rank $g$, we have a converse to the above condition when we consider $\pi_A$ of a generic curve of genus $g$. To see this, recall that by the Burnside Basis theorem $\sigma(G)=g$ if and only if $G/ \Phi(G)$ has $p$-rank $g$, where $\Phi(G)$ is the Frattini subgroup of $G$. Then since $\pi_1(D)$ 
has $p$-cohomological dimension $1$ \cite[Theorem 5.1 or Corollary 5.2] {SGA4}, it suffices to show (by \cite[I, Section 3.4, Proposition 16]{Ser94}) that the elementary abelian group $G/ \Phi(G)$ lies in $\pi_A$ of the general curve of genus $g$; but this holds since the Hasse-Witt invariant of such a curve is $g$. Also, by \cite{shaf}  the maximal pro-$p$-quotient of $\pi_1(D)$ is a free profinite $p$-group.

\subsection{Condition III}
This brings us to a condition by Nakajima concerning the augmentation ideal. 
Suppose that there exists a $G$-Galois cover $\phi:X \to D$ over $k$. 
In characteristic zero, the existence of the $G$-Galois cover $X \to D$ implies (by Chevalley-Weil \cite{CW}) that $H^0(X,\phi^*(K_D))$ is isomorphic to $k\oplus(k[G])^{g-1}$ as a $k[G]$-module (where $K_D$ is the canonical divisor on $D$). 
In characteristic $p>0$, much less is known about the $k[G]$-module structure of $H^0(X,\phi^*(K_D))$. 
However, there are some results, and these induce further conditions on the group $G$. 
For example, given a finite group $G$, let $t(G)$ denote the minimal number of generators of the augmentation ideal $I_G=\{\sum_{\sigma \in G} a_{\sigma} \sigma ~|~  \sum_{\sigma \in G} a_{\sigma} =0\} \subset k[G]$ as a $k[G]$-module.

\begin{proposition} \cite[Theorem 4]{Na84} \label{Na84}
If G lies in $\pi_A(D)$ then there exists an exact sequence of $k[G]$-modules
$$1 \to H^0(X,\phi^*(K_D)) \to k[G]^g \to I_G \to 1.$$
\end{proposition}

\begin{cor} \cite[Theorem A]{Na87}. 
If $G \in \pi_A(D)$ then $t(G) \leq g$.
\end{cor}

\noindent Also, in the case that $p$ does not divide the order of the
group, Nakajima recovers the classical Chevalley-Weil result that
$H^0(X,\phi^*(K_D))$ is isomorphic to $k\oplus(k[G])^{g-1}$
\cite[Corollary on p.\ 5]{Na84}.

\subsection{Generic $\pi_A$.} 
The conditions above are necessary for a group to lie in $\pi_A(D)$, but they leave room for other conditions that depend on more than just the genus 
of $D$. Therefore, it is useful to consider what groups occur over ``most'' curves of genus $g$. More precisely, let $\pi_A(g)$ be the set of finite groups 
$G$ for which there exists a curve $D$ of genus $g$ such that $G$ lies in $\pi_A(D)$. Proposition 4.1 of \cite{St96} implies that if $G$ lies in $\pi_A(g)$ 
then there exists an open dense subset $U_G$ of the moduli space $M_g$ of curves of genus $g$ such that for every point in $U_G$, the group $G$ 
lies in $\pi_A$ of the corresponding curve. Now we can summarize the results in this section as follows.
If a finite group $G$ lies in $\pi_A(g)$ then the following must all hold:
\begin{enumerate}
\item[I]  $G$ is generated by elements $a_1, ..., a_g, b_1, ..., b_g$, subject to the
relation $\prod_{j=1}^g [a_j , b_j ]=1$ [Prop \ref{SGA1Cor212}].
\item[II] $\sigma(G) \leq g$.
\item[III] $t(G) \leq g$. 
\end{enumerate}
\noindent Notice that for a prime-to-$p$ group, Proposition \ref{SGA1Cor212} implies that (I) is also sufficient for $G$ to lie in $\pi_A(g)$, and for a $p$-group the comments above imply that (II) is sufficient. But for groups $G$ that are neither of order prime-to-$p$ nor a power of $p$, the above are not known to provide a condition that is sufficient to imply that $G$ is in $\pi_A$ of a curve of genus $g$. Moreover, it is not immediately clear what the relationships are between these various conditions.

There are several easy comparisons to be made between Conditions (I) and (II). 
For example, any abelian $p$-group $P$ with $g<\sigma(P)<2g$ will satisfy Condition (I) but not (II). 
And similarly, any prime-to-$p$ group requiring more than $2g$ generators will satisfy Condition (II) but not (I). 
We can also ask whether there are groups satisfying Conditions (I) and (II) but which do not lie in $\pi_A(g)$. The answer is yes (see \cite[Proposition 2.5]{St98}) using Nakajima's condition (Condition (III)). 
Further work along these lines can be found in \cite{PS} and \cite{Borne}.
Additionally, in \cite{Sa} and in \cite{St96} sufficient conditions for groups to occur in $\pi_A(g)$ are given.  These conditions are expressed in terms of generators and relations for the groups.  

\section{Ramification Theory} \label{Sram}

In this section, we describe results, current research, and open problems in ramification theory.
The topic of ramification theory plays a role in the discussion of Abhyankar's 
Inertia Conjecture in Section \ref{SConj2}. 
As seen in Sections \ref{Sbouwpriesconstraint} and \ref{Sconstraint}, 
the structure of inertia groups of covers may depend on 
more subtle information about the ramification filtration.  
Throughout this chapter, $k$ is an algebraically closed field of
characteristic $p$.

Ramification theory also plays a role in understanding the structure of fundamental groups.
The reason for this is that the higher ramification groups of a Galois cover $\phi:Y \to X$ determine the genus of $Y$. 
Thus they give information about the $p$-rank of $Y$, which in turn describes the maximal 
unramified elementary abelian $p$-group cover of $Y$.

These are some of reasons we are interested in 
classifying the filtrations of higher ramification groups of wildly ramified Galois covers,
as defined in Section \ref{Shigherram}.
In particular, one can ask which ramification filtrations and jumps 
occur for covers of the affine line for a given quasi-$p$ group $G$.

\begin{question} \label{1Q2}
Suppose there exists a $G$-Galois cover $\phi : Y \to \PP^1_k$ branched at only one point with inertia $I$. 
Then what are the possible filtrations of higher ramification groups,
what are the possible sequences of jumps in the filtration, and in
particular, what is the smallest possible genus of $Y$?
\end{question}

If $G$ is a $p$-group, note that $\phi$ is totally ramified above the branch point.
If not, then the inertia group would be contained in a normal proper subgroup $N$.
The quotient of $\phi$ by $N$ would then be an unramified cover of the projective line, which is impossible since 
the \'etale fundamental group of the projective line is trivial.

In this situation, it suffices to study extensions of $L_0=k((t))$, because of Corollary~2.4 of \cite{Har80}.  That result states that there is an equivalence of categories between $G$-Galois covers of the projective line ramified only above $\infty$, and $G$-Galois \'etale covers of ${\rm Spec}\bigl(k((t))\bigr)$, given by restriction.  (An important generalization, known as the Katz-Gabber theorem, is discussed in Section~\ref{aSobstructions} below.)

In this section, we describe results for abelian $p$-groups that follow from class field theory.
Then we explain why all sufficiently large prime-to-$p$ integers can occur as the last upper jump.
Finally we describe existence results for covers with small upper jumps.

\begin{remark}
Theorem \ref{BPyesno} indicates that the answer to Question \ref{1Q2} may not be easy to formulate. 
\end{remark}

\subsection{Class field theory}

When $G$ is an abelian $p$-group, the answer to Question \ref{1Q2} is known due to class field theory.  
As mentioned above, there is only one ramification point, with inertia group equal to $G$, and it suffices to study the ramification of $G$-Galois extensions of $L_0=k((t))$ by \cite[Corollary~2.4]{Har80}.
The possible higher ramification filtrations of the inertia group can be determined using norm groups or Witt vector theory.

In general, the upper numbering is easier to analyze since it is preserved under quotients.  
While the upper jumps of an extension are a priori rational numbers, the Hasse-Arf theorem guarantees that the upper
jumps of an abelian extension are integers.
Elementary abelian $p$-group extensions are straightforward since they are fiber products of $\Z/p$-extensions.

\begin{example}
A sequence of positive integers $w_1 \leq \cdots \leq w_n$
occurs as the set of upper jumps of a $({\mathbb Z}/p)^n$-Galois extension of $L_0$
if and only if $p \nmid w_i$ for $1 \leq i \leq n$. 
\end{example}

We include the proof of this well-known example for completeness. 
 
\begin{proof} 
If $w_1 \leq \cdots \leq w_n$ is the set of upper jumps of a $({\mathbb Z}/p)^n$-Galois extension of $L_0$, 
then $w_i$ is the upper jump of one of the ${\mathbb Z}/p$-Galois quotients and thus must be prime-to-$p$.
Conversely, suppose $p \nmid w_i$ for $1 \leq i \leq n$.  Suppose $\{w_1', \ldots, w'_m\}$ is the set of distinct 
integers occurring in the sequence and let $a_i$ be the multiplicity of $w_i'$ in the sequence.
Consider the $({\mathbb Z}/p)^{a_i}$-Galois extension $\psi_i$ of $L_0$ given by the equation
\[y^{p^{a_i}} - y = x^{w_i'}.\]
It has upper jump $w_i'$ occurring with multiplicity $a_i$.  
Also the extensions $\psi_i$ are disjoint since their upper jumps are different.
Then the fiber product of $\psi_i$ is a $({\mathbb Z}/p)^n$-Galois extension of $L_0$ having upper jumps
$w_1 \leq \cdots \leq w_n$
\end{proof}

Here is the main result along these lines.

\begin{example}[\cite{Schmid}]
A sequence of positive integers $w_1 \leq \cdots \leq w_n$
occurs as the set of upper jumps of a ${\mathbb Z}/(p^n)$-Galois extension of $L_0$
if and only if $p \nmid w_1$ and, for $1 < i \leq n$, either
$w_i=pw_{i-1}$ or both $w_i > pw_{i-1}$ and $p \nmid w_i$.
\end{example}

The situation for any abelian $p$-group can be deduced from the previous two examples.
Schmid also studied some non-abelian $p$-groups, 
such as quaternion groups in characteristic $2$, using class field theory.
Necessary and sufficient conditions on the jumps of $\Z/(p^n) \rtimes \Z/m$ extensions of $L_0$ 
are determined in \cite{OP10}.

\subsection{Increasing the conductor}

Define the conductor $J$ of a Galois cover at a wild ramification point to be the largest jump in the upper numbering ramification filtration 
of inertia groups.  Let $P$ be a $p$-group.  
Consider a $P$-Galois extension $\phi: L_0 \hookrightarrow L$.  
By ramification theory  \cite[IV, Proposition 10]{Serr}, 
the last non-trivial ramification group $A$ (namely, the one with index $J$) 
is contained in the center of $P$.  Let $\overline{P}=P/A$.
Then $\phi$ factors as
\[L_0 \stackrel{\overline{\phi}}{\hookrightarrow} \overline{L} \stackrel{\pi}{\hookrightarrow} L,\]
where $\overline{L}$ is the fixed field $L^A$.  Here
$\overline{\phi}$ is a $\overline{P}$-extension and $\pi$ is an $A$-extension.

Geometrically, the extensions of $L_0$ above correspond to \'etale covers of $\hat{U}={\rm Spec}(L_0)$,
or equivalently branched covers of $\hat X = {\rm Spec}(k[[t]])$ 
ramified only over its closed point $b$.
Let $\hat{V}={\rm Spec}(L)$ and $\hat{W}={\rm Spec}(\overline{L})$.
With some abuse of notation, consider the cover
\[\hat{V} \stackrel{\pi}{\to} \hat{W} \stackrel{\overline{\phi}}{\to} \hat{U}.\]

Cohomologically, each $P$-Galois extension of $L_0$ corresponds to an element of 
the set ${\rm Hom}(\pi_1(\hat{U}), P)$, and similarly for the $\bar P$-Galois extensions.
There is then a reduction map 
\[R:{\rm Hom}(\pi_1(\hat{U}), P) \to {\rm Hom}(\pi_1(\hat{U}), \overline{P}).\]
By ideas perhaps best attributed to Serre, 
the fiber of $R$ over $\overline{\phi}$ is a principal homogeneous space for the group 
${\rm Hom}(\pi_1(\hat{U}), A)$.
This allows one to ``twist'' $\phi$ to another cover dominating $\overline{\phi}$ 
by an element $\alpha$ of ${\rm Hom}(\pi_1(\hat{U}), A)$.
The conductor of the twist is typically the maximum of the conductors of $\phi$ and $\alpha$.

These ideas yield the following result.  For a generalization to the case 
$P \rtimes \Z/m$ see \cite[Proposition 2.7]{Pr06:genus}.

\begin{proposition} \label{Pcondtwist}
Let $P$ be a $p$-group.  Every sufficiently large $J \in {\mathbb N}$ with $p \nmid J$ occurs as the conductor of a 
$P$-Galois extension of $k((t))$.
\end{proposition}

One can use Proposition \ref{Pcondtwist} and formal patching to increase the conductor at a wildly ramified point 
for a $G$-cover, even when 
$G$ is not a $p$-group or when the cover is branched at more than one point.

\begin{theorem}[\cite{Pr06:genus}, Corollary 3.3] \label{largegenus}
If $G \not= 0$ is a quasi-$p$ group and $J \in {\mathbb N}$ with $p \nmid J$ is sufficiently large, 
then there exists a $G$-Galois cover $\phi: Y \to \PP^1$ branched at only one point with conductor $J$.
\end{theorem}

\begin{theorem}[\cite{Pr06:genus}, Corollary 3.4] 
Suppose $X$ is a smooth projective irreducible $k$-curve and $B \subset X$ is a non-empty finite set of points. 
Suppose $G$ is a finite quotient of $\pi_1(X-B)$ and that $p$ divides $|G|$.  Let $g \in {\mathbb N}$.
Then there exists a $G$-Galois cover $\phi:Y \to X$ branched only above $B$ with genus $Y$ at least $g$.
\end{theorem}

\subsection{Realizing small conductors} \label{Srealizingcond}

Because of the results in the previous section, it is most interesting to study which {\it small} conductors occur
for wildly ramified Galois extensions.  
We briefly summarize the results on this topic.

As in Theorem~\ref{aTraynaud1}, let $G(S) \subset G$ be the subgroup generated by all proper quasi-$p$ subgroups $G'$ such that $G' \cap S$ is a Sylow $p$-subgroup of $G'$. The group $G$ is {\it $p$-pure} if $G(S) \not = G$.

\begin{theorem}[\cite{Pr02:cond1}, Theorem 3.5]
Let $p$ be odd.
Let $G$ be a finite $p$-pure quasi-$p$ group whose Sylow $p$-subgroups have order $p$. 
Then there exists a $G$-Galois cover $\phi: Y \to \PP^1$ branched only above $\infty$ with conductor $\sigma< 3$.
In particular, ${\rm genus}(Y) \leq 1 + |G|(p -1)/2p$.
\end{theorem}

Using Abhyankar's equations or reduction techniques, for a few quasi-$p$ groups 
it is possible to find Galois covers of the affine line 
having small conductors and non-integral conductors.  For these groups, one can be more explicit about which conductors occur.

\begin{hypothesis} \label{Hexplicitconductor}
Let $G$ be one of the following quasi-$p$ groups:\\
(i) $A_p$ (for $p \geq 5$);\\
(ii) ${\rm PSL}_2(p)$ (for $p \geq 5$), or\\ 
(iii) ${\rm PSL}_2(\ell)$ (for $p \geq 7$ and $\ell$ prime such that $p | (\ell^2 - 1)$).\\
Let $I \subset G$ be an extension of a cyclic group of prime-to-$p$ order by a $p$-group.
The group theoretic structure of $I$
and the constraints on the ramification filtrations of $I$-Galois covers of $k((t))$ 
place obvious necessary conditions on the conductor $\sigma \in {\mathbb Q}$ 
of a $G$-Galois cover $f:Y\to \PP^1_k$ branched only at $\infty$ and having inertia group $I$.
As seen in \cite[Theorems 3.2, 3.4, 3.6]{BP} and \cite[Corollary 3.6,
Remark 3.8]{obusabhy}, combined with the Hasse-Arf theorem, these conditions are that
$(p-1)\sigma/2$ is a prime-to-$p$ natural number in cases (i) and (ii)
and that $2 \sigma$ is a prime-to-$p$ natural number in case (iii).
\end{hypothesis}

\begin{theorem}[\cite{BP}, Theorem 2; \cite{obusabhy}, Corollary 4.5]
Let $G$ be one of the quasi-$p$ groups from Hypothesis \ref{Hexplicitconductor}.
Then there exists an (explicit) constant $C$, depending on $G$, such that every
$\sigma \geq C$ satisfying the obvious necessary conditions from Hypothesis \ref{Hexplicitconductor}
is the conductor of a $G$-Galois cover $f:Y\to \PP^1_k$ branched only at $\infty$.
\end{theorem}

\begin{theorem}[\cite{MP12}, Corollary 5.6] \label{Ccor2} 
Suppose $2 \leq s < p$ and ${\rm gcd}(p-1,s+1)=1$.  
Then all but finitely many rational numbers $\sigma$ such 
that $\sigma > 1$ and $(p-1)\sigma$ is a prime-to-$p$ natural number
occur as the conductor of an $A_{p+s}$-Galois cover of $\PP^1$ branched only at $\infty$.
\end{theorem}

These results are non-constructive, in that
the existence of a cover with small conductor is used to produce another cover with larger conductor. 

\subsection{Open problems}

Question \ref{1Q2} can be refined and/or generalized in myriad ways.  The most obvious generalization is 
to consider ramification filtrations of arbitrary base curves.  Here are a few other questions about 
ramification filtrations.

\subsubsection{Large non-integral conductors}

For most non-abelian $p$-groups $P$, the conditions on the ramification filtration of $P$-Galois extensions of 
$k((t))$ are not known.

Theorem \ref{largegenus} indicates that all large prime-to-$p$ integers occur as the conductor of a $P$-Galois extension.  For most non-abelian $p$-groups $P$, one can ask which sufficiently large non-integral conductors
occur for $P$-Galois extensions of $k((t))$.  We ask this basic algebraic question:

\begin{question} \label{Qnonintegralconductor}
Given a $p$-group $P$, what is the smallest positive integer $N_P$ such that all but finitely many 
conductors of $P$-Galois covers of $k((t))$ are contained in $\frac{1}{N_P} {\mathbb Z}$?
\end{question}

Using Herbrand's Theorem \cite[IV, Section 3]{Serr}, see Definition \ref{Dherbrand},
one sees that $N_P$ divides $|P|$.
But $N_P$ can be much smaller than $|P|$.
For example, $N_P=1$ if $P$ is abelian by the Hasse-Arf Theorem.

Question \ref{Qnonintegralconductor} can be generalized to quasi-$p$ group covers of the affine line.

\begin{question} \label{Qnonintegralconductor2}
Given a quasi-$p$ group $G$, what is the smallest positive integer $N_G$ such that all but finitely many 
conductors of $G$-Galois covers $\phi: Y \to \PP^1_k$ branched
only at $\infty$ are contained in $\frac{1}{N_G} {\mathbb Z}$?
\end{question}

We remark that the number $N_G$ is well-defined.  
For fixed $G$, there are finitely many possibilities for the inertia groups of $\phi$ and $N_G$ must divide the 
least common multiple of the orders of these.
However, it is not clear if Question \ref{Qnonintegralconductor2} will have a clean answer
when $G$ is not a $p$-group.  It is also not clear if the current techniques of formal patching and semi-stable reduction 
of covers give enough information to answer Question \ref{Qnonintegralconductor2}.

\subsubsection{Moduli problems for covers with fixed ramification filtration}

Suppose $\phi:Y \to X$ is a ramified $G$-Galois cover of curves.  
By Lemma \ref{1SRH}, the genus $g_Y$ of $Y$ depends on the genus $g_X$ of $X$, the order of $G$,
and the ramification filtrations above all the branch points.  
For fixed $g_Y$, $g_X$ and $|G|$, there can be several choices for how to divide the degree of the different 
among the branch points.  

For example, when $G={\mathbb Z}/p$ and $X = \PP^1_k$, then 
\[2g_Y=(p-1)(-2 + \sum_{\xi \in B} (j_\xi +1)),\]
where $j_\xi$ is the lower jump of $\phi$ above the point $\xi$.

There are finitely many ways to divide the degree of the different among the branch points.  
Suppose we fix one of these, and call it a ramification configuration, denoted $\eta$.  
Then one can ask about the geometry of 
the moduli space of covers $\phi$ with ramification configuration $\eta$.  
For example, is it irreducible?  What is the dimension?

In \cite[Theorem 1.1]{PZ12:artschprank}, the authors prove that the moduli space of ${\mathbb Z}/p$-Galois 
covers of $\PP^1_k$ with fixed ramification configuration is irreducible and compute its dimension in terms of the ramification configuration.  A natural question is whether \cite[Theorem 1.1]{PZ12:artschprank}
can be generalized to $p$-group covers of curves of arbitrary genus.

\begin{question}
Fix a $p$-group $P$, integers $g_Y$ and $g_X$, and a ramification configuration $\eta$.  
Consider the moduli space ${\mathcal M}_\eta$ 
of $P$-Galois covers $\phi:Y \to X$ with $Y$ of genus $g_Y$ and $X$ of genus $g_X$ with ramification 
configuration $\eta$.  Then is ${\mathcal M}_\eta$ irreducible?
\end{question}

\section{The lifting problem, the local lifting problem, and the Oort conjecture}\label{aSoort}

In the earlier parts of this paper, we have discussed much about constructing branched 
$G$-covers of algebraic curves over $k$ in characteristic $p$.  In particular, when $p$ does not divide the order of $G$, such covers are in one-to-one correspondence with 
branched $G$-covers of algebraic curves in characteristic zero, once a lift of the branch locus is given  (cf.\ Section \ref{Sfundgroups}); in fact, they arise from the mod $p$ reduction of the corresponding cover in characteristic zero.  However, when $p$ divides the order of $G$,
we have seen that things get more complicated.  In particular (cf.\ Section \ref{Sinertiareduction}), $G$-covers in characteristic zero can have bad reduction, and 
thus do not necessarily give rise to branched covers in characteristic $p$.  Furthermore, there exist many $G$-covers in characteristic $p$ that have no
obvious analogues in characteristic zero.  For example, by Abhyankar's Conjecture \ref{dAbhConj}, there exist $G$-covers of $\proj^1_k$, branched at
one point, whenever $G$ is a finite simple group whose order is divisible by $p$.  In general, it is not easy to see when these covers come from reducing
covers in characteristic zero (clearly such a cover cannot be the reduction of a $G$-cover of $\proj^1$ branched at one point in 
characteristic zero when $G$ is nontrivial!).  This
section addresses this question, namely, \emph{when is a branched $G$-cover of  
$k$-curves the reduction of a $G$-cover of curves in characteristic zero?}.  

The exposition here is necessarily short and limited.  
Much more thorough and lengthy accounts are given in \cite{Ob:ll} and \cite{Ob:lc}.

Throughout Section \ref{aSoort}, $G$ always denotes a finite group, and the field $k$ is assumed to be algebraically
closed of characteristic $p > 0$.  
A $G$-cover of curves over a field is assumed to involve only projective, smooth, geometrically connected curves.  

\subsection{The (global) lifting problem}\label{Sglobal}

We state the problem above more precisely:

\begin{question}[The lifting problem for $G$-covers of curves]\label{aQlifting}
Let $f: Y \to X$ be a $G$-cover of curves over $k$.   
Does there exist a characteristic zero discrete valuation ring $R$ with residue field $k$, along with a flat morphism $f_R: Y_R \to X_R$ of
relative smooth projective curves over $\Spec R$ such that
\begin{enumerate}
\item The special fiber of $f_R$ is $f$, and
\item There is an action of $G$ on $Y_R$ by $R$-automorphisms such that $f_R$ is the quotient map of this action?
\end{enumerate}
\end{question}
If the answers to (1) and (2) are ``yes," then we say that the cover $f$ \emph{lifts to characteristic zero} (or \emph{lifts over $R$}), and that $f_R$ is a 
\emph{lift} of $f$ (with $G$-action).  

\begin{remark} The generic fiber of $f$ is a $G$-cover $f_K: Y_K \to X_K$ over the fraction field $K$ of $R$.  Note
also that it follows from standard deformation theory techniques in \cite[III]{SGA1} that \emph{individual curves} lift to characteristic zero.  If $X_R$ is
such a lift of $X$, with generic fiber $X_K$, then one can solve the lifting problem by exhibiting a $G$-cover 
$f_K: Y_K \to X_K$ such that the normalization of $X_R$ in the function field of $Y_K$ has special fiber $f$.  In this case, we say that $f$ is the
\emph{reduction} of $f_K$.  
\end{remark}

The following is a major result of \cite{SGA1}.  We do not discuss the original proof here, because it will end up being an easy consequence of the 
\emph{local-global principle} (Theorem \ref{aTlocalglobal}).

\begin{thm}\label{aTtamelifting} 
Tame $G$-covers over $k$ (in particular, \'{e}tale $G$-covers over $k$) lift to characteristic zero, and we can take $R = W(k)$, the Witt vectors over $k$.  
\end{thm}

Theorem \ref{aTtamelifting} shows that every tame $G$-cover is the (good) reduction of a $G$-cover in characteristic zero.  
However, if $G$ does not have 
prime-to-$p$ order, then there are in general $G$-covers in characteristic zero that have bad reduction to characteristic $p$.  
Furthermore, the argument in \cite{SGA1} shows that if $f: Y  \to X$ is a tame $G$-cover of 
smooth, projective, connected $k$-curves and $X_R$ is a lift of $X$ to $R$, then every lift $D_R$ of the branch locus $D$ of $f$ to $X_R$ 
gives a unique lift of $f$ to a $G$-cover $f_R: Y_R \to X_R$ (branched over $D_R$) as above.  For each $R$-point $x$ of $D_R$, the branching index
of $f$ above the closed fiber of $x$ is equal to the branching index of $f_K$ above the generic fiber of $x$.

When one phrases this in terms of fundamental groups, one recovers Proposition \ref{SGA1Cor212}. 
Let $U = X - \Delta$, let $U_R = X_R - \Delta_R$, 
let $U_K$ be the generic fiber of $U_R$, and let $U_{\ol{K}}$ be the base change of $U_K$ to the algebraic closure $\ol{K}$.  Suppose that $X$ has 
genus $g$ and $\Delta$ consists of $r$ points.  Then the $p$-tame fundamental group of
$U_{\ol{K}}$ is isomorphic (say, under some isomorphism $\alpha$) to 
the group $\hat{F}_{g,r}^{p-tame}$ of Section \ref{Sproj}.  Proposition \ref{SGA1Cor212} says that we have a surjection 
$$pr: \hat{F}_{g,r}^{p-tame} \to \pi_1(U)^{t}.$$ 

We can see the surjectivity of $pr$ as follows: An element $s$ of $\pi_1(U)^{t}$ is a compatible system of automorphisms of all finite \'{e}tale Galois
covers of $U$ that extend to tame covers of $X$.  Each of these covers lifts uniquely to a Galois cover of $X_{\ol{K}}$ that restricts to a 
$p$-tame \'{e}tale cover of $U_{\ol{K}}$.  Thus $s$ comes from a compatible system of automorphisms of those $p$-tame \'{e}tale Galois covers of
$U_{\ol{K}}$ which have good reduction.  Such a system can be extended (in many ways) 
to a compatible system of automorphisms of \emph{all} $p$-tame \'{e}tale Galois covers of $U_{\ol{K}}$.  Such a system corresponds (under $\alpha$)
to an element  $s' \in \hat{F}_{g,r}^{p-tame}$ mapping to $s$. 

What about possibly wildly ramified $G$-covers?  The following example shows that the lifting problem need not always have a solution.

\begin{example}\label{aEpxp}
Let $G = \ints/p \times \ints/p$, and consider the $G$-cover $f: \proj^1_k \to \proj^1_k$ given by taking the quotient of $\proj^1_k$ by
an embedding of $G$ into ${\rm PGL}_2(k)$. If $f$ 
lifts over a characteristic zero discrete valuation ring $R$, 
the generic fiber $f_K$ of the lift $f_R$ must be a $G$-cover of genus zero curves, and after possibly extending $R$, we may assume
that $f_K$ is a $G$-cover of $\proj^1_K$'s.  But if $p > 2$, then $G$ cannot act on $\proj^1_K$, as it does not embed into ${\rm PGL}_2(K)$.
So $f$ cannot lift to characteristic zero.
\end{example}

\begin{remark}
In Example \ref{aEpxp}, the cover $f$ will be (wildly) ramified at exactly one point, with inertia group $\ints/p \times \ints/p$.
\end{remark}
  
\subsection{The local-global principle}

As phrased in Question \ref{aQlifting}, the lifting problem is a global question about algebraic curves.  Indeed, the obstruction to lifting in 
Example \ref{aEpxp} is global in nature.  It is thus somewhat surprising that, thanks to the \emph{local-global principle} below, the lifting problem can be 
studied in a completely local way, without losing any information. In the sequel, a Galois extension of rings means a finite extension of integrally
closed integral domains which is Galois on the level of fraction fields.

\begin{theorem}[Local-global principle]\label{aTlocalglobal}
Let $f: Y \to X$ be $G$-cover of $k$-curves, and 
let $x_1, \ldots, x_s \in X$ be the branch points of $f$.  For each $i$, $1 \leq i \leq s$,
let $G_i$ be the inertia group at some point $y_i \in Y$ above $x_i$, and let $\hat{\mc{O}}_{Y, y_i}/\hat{\mc{O}}_{X, x_i} \cong k[[z_i]]/k[[t_i]]$ be the 
corresponding $G_j$-extension of complete local rings.  If $R$ is a finite extension of $W(k)$, then $f$ lifts over $R$ if and only if
each $k[[z_i]]/k[[t_i]]$ lifts over $R$ (that is, iff there is a $G_i$-extension $R[[Z_i]]/R[[T_i]]$ reducing to $k[[z_i]]/k[[t_i]]$ --- our convention is that
capital letters reduce to the respective lowercase letter). 
\end{theorem}

We note that the ``only if" direction above is easy to show, because if $f_R: Y_R \to X_R$ is a lift, we can take $R[[Z_i]]/R[[T_i]]$ simply to be the
extension $\hat{\mc{O}}_{Y_R, y_i}/\hat{\mc{O}}_{X_R, x_i}$.  For the much more difficult ``if" direction, proofs based on patching are given by Garuti in
\cite[Section 3]{Ga:pr} and Green-Matignon in \cite{GM:lg}, and a proof based on deformation theory is given in \cite[Section 3]{BM:df}.

\subsection{The local lifting problem}\label{aSlocallifting}

Using Theorem \ref{aTlocalglobal}, we have reduced Question \ref{aQlifting} (at least over complete discrete valuation rings with residue field $k$) 
to the following \emph{local lifting problem}: 

\begin{question}[Local lifting problem]\label{aQlocallifting}
Let $k[[z]]/k[[t]]$ be a $G$-Galois extension.  Does there exist a complete characteristic zero discrete valuation ring $R$ with residue field $k$,
and a $G$-Galois extension $R[[Z]]/R[[T]]$ lifting $k[[z]]/k[[t]]$?
\end{question}
If such a lift exists, we say that $k[[z]]/k[[t]]$ \emph{lifts to characteristic zero} (or \emph{lifts over $R$}).
By Example \ref{aEpxp} and the local-global principle, we know that the local lifting problem cannot always have a solution 
(indeed, the $\ints/p \times \ints/p$-extension corresponding to the unique ramification point in Example \ref{aEpxp} cannot lift to characteristic zero).

\begin{remark}\label{aRinertiastructure}
By Lemma \ref{1SRH} (1), the group $G$ will 
automatically be of the form $P \rtimes \ints/m$, where $P$ is a $p$-group and $p \nmid m$.  This is one advantage of the local formulation: we can restrict to a 
smaller class of groups. (In particular, all the groups we consider are solvable!)
\end{remark}

Now, in the simplest case, when $G \cong \ints/m$ with $p \nmid m$, it is easy to see that we can solve the local lifting problem 
(and we can even take $R = W(k)$). Indeed, up to a change of variable,
the only $\ints/m$-extension of $k[[t]]$ comes from extracting an $m$th root of $t$, and this simply lifts to the extension of $R[[T]]$ coming from 
extracting an $m$th root of $T$ (note that $W(k)$ contains the $m$th roots of unity).  Via the local-global principle, this gives another proof of tame
lifting (Theorem \ref{aTtamelifting}).
Below, we will discuss both aspects of the local lifting problem: showing that certain extensions lift, and identifying obstructions to lifting.

\subsubsection{The Oort conjecture}

A major driver of research on the local lifting problem over the last couple of decades has been the \emph{Oort conjecture} (now a theorem of
Obus, Wewers and Pop), stated originally in \cite{Oo:ac}. 

\begin{theorem}[(Local) Oort conjecture, \cite{OW:ce}, \cite{Po:lc}]\label{aToort}
Every cyclic $G$-extension $k[[z]]/k[[t]]$ lifts to characteristic zero.
\end{theorem}

\begin{remark}
One reason to look at cyclic extensions is because tame extensions of $k[[t]]$ are always cyclic and lift to characteristic zero.  Somehow, one thinks that
cyclic wild extensions should not be as far off from the tame case as arbitrary wild extensions are.
\end{remark}

As was mentioned above, the Oort conjecture is easy when $p \nmid |G|$.  In fact, one can readily reduce to the case of $G \cong \ints/p^n$ (see, e.g.,
\cite[Proposition 6.3]{Ob:ll}).  The case of $G \cong \ints/p$ was proven by Sekiguchi, Oort, and Suwa in \cite{SOS} in the global context.  
The proof is much easier in the local context, and we sketch it here. 

\begin{proof}[Sketch of proof of the Oort conjecture for $\ints/p$]
We can take $R= W(k)[\zeta_p]$. 
Any $\ints/p$-extension $k[[z]]/k[[t]]$ 
is given (after a possible change of variable) by taking the normalization of $k[[t]]$ in the extension of $k((t))$ given by extracting 
a root of $y^p - y = t^{-u_1}$, where $u_1$ is the upper jump of the extension (see Section \ref{Shigherram}).  Let $\lambda = \zeta_p - 1$, which is a 
uniformizer of $R$.  Also, $v(\lambda^{p-1} + p) > 1$.  
Consider the integral closure of $R[[T]]_{(\lambda)}$ in the Kummer extension $L$ of $\Frac(R[[T]])$ 
given by $$W^p = 1 + \lambda^p T^{-u_1}.$$  Making the substitution $W = 1 + \lambda Y$, we obtain
$$(\lambda Y)^p + p\lambda Y  + o(p^{p/(p-1)})= \lambda^p T^{-u_1},$$ where $o(p^{p/(p-1)})$ represents terms with coefficients of valuation
greater than $\frac{p}{p-1}$.  This reduces to $y^p - y = t^{-u_1}$, which gives the correct extension of $k((t))$.  
So we will be done if we can show that the integral closure $A$ of $R[[T]]$ in $L$ 
is in fact isomorphic to a power series ring in one variable over $R$.

By \cite[I, Theorem 3.4]{GM:lg}, it suffices to check that the degree $\delta_s$ of the different of $k[[z]]/k[[t]]$ is equal to the degree $\delta_{\eta}$ 
of the different of 
$(A \otimes_R K) / (R[[T]] \otimes_R K)$, where $K = \Frac(R)$.  One calculates $\delta_s = (u_1+1)(p-1)$, using results of Section \ref{Shigherram}. 
Furthermore, one calculates $\delta_{\eta} = (u_1 + 1)(p-1)$, as $R[[T]] \otimes_R K$ has exactly $u_1 + 1$ (totally) ramified maximal ideals 
($(T)$ and $(T-a)$ for each $u_1$th root $a$ of $-\lambda^p$).  Since the ramification is tame, each of these contributes $p-1$ to the 
degree of the different.
\end{proof}

The proof of the Oort conjecture for $G = \ints/p^2$ is due to Green and Matignon (\cite{GM:lg}).  It is much more complicated than the proof of the
$\ints/p$ case, but it follows the same rough outline:
\begin{enumerate}
\item Let $R = W(k)[\zeta_{p^2}]$.
\item Parameterize all $\ints/p^2$-extensions of $k[[t]]$ explicitly via taking the normalization in extensions of $k((t))$ 
(this uses \emph{Artin-Schreier-Witt theory}).
\item For each such extension, find an explicit Kummer extension 
$L$ of $\Frac(R[[T]])$ such that, if $\pi$ is a uniformizer of $R$, then the normalization of $R[[T]]_{(\pi)}$ in $L$ reduces to the correct extension of 
$k((t))$.
\item Show, by comparing differents, that taking the normalization of $R[[T]]$ in $L$ in fact gives a lift.
\end{enumerate}

The most difficult part above is part (3).  The extensions in \cite{GM:lg} are inspired by the \emph{Kummer-Artin-Schreier-Witt theory}
(also known as the \emph{Sekiguchi-Suwa theory}, see \cite{SS1} and \cite{SS2}).  
This theory gives an explicit exact sequence of group schemes over 
$R = W(k)[\zeta_{p^n}]$ that ``bridges the gap" between the Kummer exact sequence (of degree $p^n$) on the generic fiber and the 
Artin-Schreier-Witt exact sequence (of degree $p^n$) on the special fiber.  Unfortunately, the equations involved become somewhat intractable to
work with once $n \geq 3$.

In contrast to the above proofs, the proof of the Oort conjecture for $\ints/p^n$ with $n \geq 3$ given by Obus, Wewers, and Pop in
\cite{OW:ce} and \cite{Po:lc} is significantly less explicit.  Roughly, instead of taking the
perspective of starting with a $\ints/p^n$-extension in characteristic $p$ and trying to lift it to characteristic zero, one instead tries to write down
a (Kummer) $\ints/p^n$-extension $L/\Frac(R[[T]])$ for some $R$ such that the normalization of $R[[T]]$ in this extension reduces to \emph{some}
Galois extension of $k[[t]]$.  This is achieved by measuring the failure of the reduction to be Galois, and undergoing an iterative deformation process
of $L/\Frac(R[[T]])$ to reduce this failure measurement to zero.  Once one has a lift of one $\ints/p^n$-extension of $k[[t]]$ (which can be chosen to have
minimal jumps in the higher ramification filtration, in some sense), then one can get lifts of many more. 

In particular, \cite{OW:ce} shows that 
any $\ints/p^n$-extension $k[[z]]/k[[t]]$ with \emph{no essential ramification} lifts to characteristic zero (where ``no essential ramification" means that if 
$u_1 < \cdots < u_n$ are the successive jumps in the higher ramification filtration for the upper numbering of $k[[z]]/k[[t]]$, and $u_0 := 0$, then
$u_i < u_{i-1} + p$ for all $i \geq 1$).  The work of \cite{Po:lc} shows that all $\ints/p^n$-extensions of $k[[t]]$ can be equicharacteristically deformed, 
in some sense, to extensions that have no essential ramification, and that a lift of the deformed extension 
gives rise to a lift of the original extension.  This completes the proof of the Oort conjecture.

\begin{remark}
The methods of \cite{OW:ce} and \cite{Po:lc} are not fully explicit.  For instance, in \cite{OW:ce}, one ends up only knowing the 
coefficients of the rational function in $\Frac(R[[T]])$ giving rise to the Kummer extension within some error (measured using the ultrametric of $R$).  In
particular, one does not get any good control over the ring $R$ itself (although in \cite{Po:lc}, it is shown that, for any $N \in \nats$, 
there exists some $R/W(k)$ finite that works for all $\ints/p^n$-extensions of $k[[t]]$ with degree of different less than $N$).
\end{remark}

\subsubsection{Obstructions and local Oort groups}\label{aSobstructions}
A \emph{local Oort group} is a group $G \cong P \rtimes \ints/m$, where $P$ is a $p$-group and $p \nmid m$, such that the local lifting problem
has a solution for all $G$-extensions.  The proof of the Oort
conjecture shows that all cyclic groups are local Oort groups (or
``local Oort'' for short).  In fact, it is known
due to Bouw and Wewers (\cite{BW:ll}) that $D_p$ is local Oort for all odd $p$, due to Pagot (\cite{Pa:ev}) that $\ints/2 \times \ints/2$ is
local Oort for $p = 2$, due to Weaver (\cite{We:D4})) that
$D_4$ is local Oort for $p = 2$, due to Bouw (unpublished) and
Obus (\cite{Ob:A4}) that $A_4$ is local Oort for $p = 2$, and due to
Obus (\cite{Ob:go}) that $D_9$ is local Oort for $p = 3$.

On the other hand, 
we have seen that $\ints/p \times \ints/p$ is not a local Oort group when $p > 2$.  In fact, there is a rather limited list of local Oort groups.  Most
groups not on this list are ruled out via the so-called \emph{Katz-Gabber-Bertin obstruction} (or \emph{KGB obstruction}), 
due to Chinburg, Guralnick, and Harbater (\cite{CGH:ll}), which we describe below.

Let $k[[z]]/k[[t]]$ be a $G$-extension, where $G \cong P \rtimes \ints/m$, with $P$ a $p$-group and $p \nmid m$.  The theorem of Katz and Gabber
(\cite[Theorem 1.4.1]{katzgabber}) states that there exists a unique $G$-cover $Y \to \proj^1_k$ that is \'{e}tale outside $t \in \{0, \infty\}$, tamely ramified of 
index $m$ above $t = \infty$, and
totally ramified above $t = 0$ such that the extension of complete local rings at $t=0$ is given by $k[[z]]/k[[t]]$.  This is called the 
\emph{Katz-Gabber cover associated to $k[[z]]/k[[t]]$}.  By the local-global principle,
the $G$-cover $f: Y \to \proj^1_k$ lifts to characteristic zero iff the extension $k[[z]]/k[[t]]$ does.  

Let $R/W(k)$ be finite, and let $K = \Frac(R)$.  Suppose $f_R: Y_R \to \proj^1_R$ is a lift of $f$ to characteristic zero.  
Since genus is constant in flat families, the genus of $Y$ is equal to that of $Y_K := Y_R \times_R K$.  
Furthermore, if $H \leq G$ is any subgroup, then $Y_R/H$ is a lift of $Y/H$ with generic fiber $Y_K/H$, and the genus of $Y/H$ is equal to that of
$Y_K/H$.  

\begin{def}\label{aDKGB}
Let $k[[z]]/k[[t]]$ be a $G$-extension with associated Katz-Gabber $G$-cover $Y \to \proj^1_k$.  Then the \emph{KGB obstruction 
vanishes for $k[[z]]/k[[t]]$} if there exists a $G$-cover $X \to \proj^1$ over a field of characteristic zero such that, for each subgroup $H \subseteq G$, the
genus of $Y/H$ is equal to the genus of $X/H$.
\end{def} 

If $k[[z]]/k[[t]]$ lifts to characteristic zero, its KGB obstruction must vanish, 
since the associated Katz-Gabber cover lifts.  If $G \cong P \rtimes \ints/m$ as above and
the KGB obstruction vanishes for \emph{all} $G$-extensions $k[[z]]/k[[t]]$, then $G$ is called a \emph{KGB group} for $k$.  Thus any local Oort group
is a KGB group.

\begin{remark}
Note that, even though the KGB obstruction references global objects, its inputs are local.  Indeed, the obstruction can be phrased in an equivalent 
local form, without any reference to Katz-Gabber covers, by comparing degrees of differents of subextensions of $k[[z]]/k[[t]]$ 
to possible degrees of differents of branched covers of $\Spec
R[[T]]$, see \cite[Remark 4.6]{Ob:lc}. 
\end{remark}

\begin{example}
Consider the cover $f: \proj^1_k \to \proj^1_k$ from Example \ref{aEpxp}.  This has one branch point, and is totally ramified above the point.  The 
corresponding extension $k[[z]]/k[[t]]$ of complete local rings does not lift to characteristic zero.  One sees this by noting that $f$ itself is the Katz-Gabber
cover corresponding to $k[[z]]/k[[t]]$ (in the $p$-group special case of \cite[Corollary~2.4]{Har80}.  It is a genus zero $\ints/p \times \ints/p$-cover of a genus zero curve.
The KGB obstruction does not vanish for $k[[z]]/k[[t]]$ when $p > 2$, as there does not exist a $\ints/p \times \ints/p$-cover of genus zero curves in 
characteristic zero.
\end{example}

The following classification of all the KGB groups is due to Chinburg, Guralnick, and Harbater.

\begin{theorem}[\cite{CGH:ll}, Theorem 1.2]\label{aTKGB}
The KGB groups for $k$ consist of the cyclic groups, the dihedral group $D_{p^n}$ for any $n$, the group $A_4$ (for $\text{char}(k) = 2$),
and the generalized quaternion groups $Q_{2^m}$ of order $2^m$ for $m \geq 4$ (for $\text{char}(k) = 2$).
\end{theorem}

\begin{remark}
There is another obstruction, called the \emph{Bertin obstruction} (\cite{Be:ol}), that we do not discuss here. 
It was proven in \cite{CGH:ll} to be strictly weaker than the KGB obstruction.
\end{remark}

\begin{remark}\label{aRhurwitz}
There is yet another obstruction, called the \emph{Hurwitz tree obstruction}, discovered by Brewis and Wewers (\cite{BrewisWewers}).  We do not describe it
here, but we note that it applies to the groups $Q_{2^m}$ from Theorem \ref{aTKGB}.  That is, there exist $Q_{2^m}$-extensions with non-vanishing
Hurwitz tree obstruction.
\end{remark}

\begin{remark}\label{aRalmostKGB}
According to Theorem \ref{aTKGB}, the quaternion group $Q_8$ is \emph{not} a KGB group for $p=2$.  However, it is a so-called ``almost KGB group"
(\cite{CGH:ll}), which means that all $Q_8$-extensions have vanishing KGB obstruction as long as the smallest positive ramification jump is sufficiently large.  It is open whether $Q_8$ (and, indeed $Q_{2^m}$ for $m \geq 4$) is an ``almost local Oort group."  That is, we do not know 
whether all $Q_{2^m}$-extensions with sufficiently large smallest ramification jump lift to characteristic zero for $m \geq 3$.
\end{remark}

From Theorem \ref{aTKGB} and Remark \ref{aRhurwitz}, we see that the only possible local Oort groups are cyclic groups, dihedral groups $D_{p^n}$,
 and $A_4$ in characteristic $2$.  Of these, all are known to be local
 Oort groups except for the groups $D_{p^n}$ when $n > 1$ (and in
 fact $D_4$ and $D_9$ are local Oort).

\subsubsection{Weak Oort groups and GM-groups}  

Less strict than the property of being a local Oort group is the property of being \emph{weak local Oort group}.
A group $G$ (of the form $P \rtimes \ints/m$ as above) is called a weak local Oort group (for $k$) 
if there \emph{exists} a $G$-extension $k[[z]]/k[[t]]$ that lifts to characteristic zero.  So all local Oort groups are weak local Oort groups.  
Several more groups have been shown to be weak local Oort.  For instance, Matignon (\cite{Ma:pg}) showed that $(\ints/p)^n$ is a weak local 
Oort group for all $p$ and $n$.  Weaver's proof that $D_4$ is local
Oort for $2$ depends fundamentally on a result of Brewis (\cite{Br:D4}) that $D_4$ is weak local Oort for $2$.  Lastly, it has been shown by 
Obus (\cite{Ob:go}) that the group $\ints/p^n \rtimes \ints/m$ is weak local Oort if and only if the conjugation action of 
$\ints/m$ on $\ints/p^n$ is faithful or trivial (the case $n=1$ is proven in \cite{BW:ll}).

The main negative result on weak local Oort groups is the following:

\begin{theorem}[\cite{CGH:ll}]\label{aTnotweak}
Let $G \cong P \rtimes Y$, where $P$ is a $p$-group and $Y$ is a cyclic group of order not divisible by $p$. Let $B \subseteq Y$ be the maximal subgroup of order dividing $p-1$.  If $G$ is weak local Oort, then the following two properties hold:
\begin{enumerate}
\item For all nontrivial $y \in Y$, the centralizer of $y$ in $P$ is cyclic and equal to the centralizer of $Y$ in $P$.
\item There exists a character $\chi: B \to \ints_p^{\times}$ such that, for all cyclic subgroups $T$ of $P$ with trivial prime-to-$p$ centralizer, we have
$$xyx^{-1} = y^{\chi(x)}$$ for $y \in T$ and $x \in B$.
\end{enumerate}
\end{theorem}
In particular, if $G$ has an abelian subgroup that is neither cyclic or a $p$-group, then it cannot be a weak local Oort group (this was an earlier result
of Green and Matignon, see \cite{Gr:af} for a proof).  A group satisfying the criteria of Theorem \ref{aTnotweak} is called a \emph{GM-group}.
 
\subsubsection{Open problems}
We collect and comment on some open problems:

\begin{question}\label{aQeffectiveoort}
The \emph{arithmetic version} of the Oort conjecture states that a cyclic $\ints/r$-extension $k[[z]]/k[[t]]$ should lift over 
$W(k)[\zeta_r]$.  Does this hold?
\end{question}
As we have seen, the proofs of \cite{SOS} and \cite{GM:lg} mentioned above prove this stronger version when $v_p(r) \leq 2$ 
by exhibiting explicit liftings.  Is this true when $v_p(r) > 2$?

\begin{question}\label{aQcyclicsylow}
Is every KGB group other than those of the form $Q_{2^n}$ a local Oort group?
\end{question}
As was mentioned before, the only groups that remain to be checked are
those of the form $D_{p^n}$ for $n \geq 2$.  The second author conjectures
the following:

\begin{conj}\label{aCmyconj}
If $G$ has a cyclic $p$-Sylow subgroup, then for any $G$-extension $k[[z]]/k[[t]]$, the KGB obstruction is the only obstruction to lifting.
\end{conj}
Notice that Conjecture \ref{aCmyconj} would imply that $D_{p^n}$ is a local Oort group for \emph{odd} $p$. 

Similarly to Question \ref{aQcyclicsylow}, we may ask:
\begin{question}\label{aQGM}
Is every GM-group a weak local Oort group?
\end{question}
A positive answer to Question \ref{aQGM} would show, for instance, that every $p$-group is a weak local Oort group.

\begin{question}\label{aQoortdef}
What can we say about the versal deformation ring of a given $G$-extension $k[[z]]/k[[t]]$?  
\end{question}
Question \ref{aQoortdef} in the case $G = \ints/p$ 
is the subject of \cite{BM:df}, although the authors of \cite{BM:df} gave a complete explicit description only when the jump in the higher ramification filtration is~$1$.  Note that this is an interesting question, even when the given extension is not liftable to characteristic zero.  For instance, might the extension
be liftable to characteristic $p^{\ell}$ for some $\ell$ (i.e., might
it lift to some $G$-extension $A[[z]]/A[[t]]$ where $A$ is a local artinian
ring of characteristic $p^l$ with residue field $k$)?  The paper \cite{CM:rr} of Cornelissen and M\'{e}zard states that, if the extension is 
\emph{weakly ramified} (i.e., has trivial second higher ramification group), then liftability to characteristic $p^2$ is equivalent to liftability to characteristic
zero.


\bibliographystyle{plain}                                    

\end{document}